\documentclass[towside]{article}
\usepackage[left=7mm,right=7mm]{geometry} 
\usepackage{vmargin} 
\usepackage{xr}
\usepackage[numbers]{natbib}
\usepackage{fancyhdr}
\usepackage{amsmath}
\usepackage{mathrsfs}

\usepackage{mathabx}

\usepackage{amsthm}

\usepackage{amsfonts}
\usepackage{amssymb}
\usepackage{graphicx}
\usepackage{placeins}
\usepackage{algorithm2e}
\usepackage{caption}
\usepackage{array}
\usepackage{multicol}
\usepackage{color}
\usepackage{bbm}
\usepackage{mhequ}
\usepackage{enumerate}
\usepackage[font={small,it}]{caption}

\usepackage{accents}
\renewcommand{\underbar}[1]{\underaccent{\bar}{#1}}

\theoremstyle{plain}

\newtheorem{theorem}{Theorem}[section]

\newtheorem{corollary}[theorem]{Corollary}

\newtheorem{lemma}{Lemma}[section]

\theoremstyle{definition}

\newcommand{\PI}{\pi} 
\renewcommand{\d}{{\text{d}}}
\newcommand{\cN}{{\mathcal{N}}}
\newcommand{\mc}[1]{\mathcal{#1}}
\newcommand{\bb}[1]{\mathbb{#1}}
\newcommand{\bs}[1]{\boldsymbol{#1}}

\newcommand\ci{\perp\!\!\!\perp}

\newcommand{\Exp}{\textnormal{Exponential}}

\newcommand{\jwtime}{\kappa_{\Delta}} 

\newcommand{\gamdy}{\gamma_\d} 

\newcommand{\THETA}{\bbR} 
\newcommand{\bbH}{{\bb H}}
\newcommand{\ck}{{\mathcal{K}}}   

\newcommand{\Fr}{\mathsf{Fr}}

\newcommand{\xst}{\xi_{s^*}}

\newcommand{\vpL}{\varphi_\Lambda}%
\newcommand{\msv}{\mathsf{MSV}}
\newcommand{\Yt}[1]{Y_{t_{#1}}(x_{#1})}
\newcommand{\distd}{\mathrel{\stackrel{\mc{D}}{=}}}
\newcommand{\iN}{_{1\le i\le n}}
\newcommand{\jJ}{_{j\in J}}

\def \be{\begin{equs}}
\def \ee{\end{equs}}
\newcommand{\Ymn}{\vee\iN Y_i(x,t)}
\newcommand{\Ystmn}{\vee\iN Y^*_i(x,t)}
\newcommand{\wij}{\omega^{(i)}_j}
\newcommand{\Ni}{\cN^{(i)}}

\DeclareMathOperator{\M}{M}
\DeclareMathOperator{\AD}{AD}
\DeclareMathOperator{\KS}{KS}
\DeclareMathOperator{\var}{var}
\newcommand{\wh}[1]{\widehat{#1}}

\newcommand{\km}{\thinspace{\text{km}}}

\newcommand{\bbN}{{\mathbb{N}}}

\newcommand{\bbR}{{\mathbb{R}}}

\newcommand{\cA}{{\mathcal{A}}}

\newcommand{\cH}{{\mathcal{H}}}

\newcommand{\cP}{{\mathcal{P}}}

\newcommand{\cX}{{\mathcal{X}}}

\renewcommand{\P}{{\mathsf{P}}} \newcommand{\E}{{\mathsf{E}}}

\newcommand{\Po}{\mathsf{Po}}

\newcommand{\hide}[1]{}
\newcommand{\argmax}{\mathop{\mathrm{argmax}}}
\newcommand{\argmin}{\mathop{\mathrm{argmin}}}

\newcommand{\iid}{\mathrel{\mathop{\sim}\limits^{\mathrm{iid}}}}

\newbox\asbox
\setbox\asbox=\hbox{\vrule height 15pt depth3.5pt width0pt}
\def\astrut{\relax\ifmmode\copy\strutbox\else\unhcopy\strutbox\fi}
\newdimen\bsigdep
\def\bSig{{\setbox0\hbox{$\Sigma$}\bsigdep=1\dp0\advance\bsigdep by
 .2\ht0 \rlap{\kern.3\wd0\vrule height1.2\ht0 depth1\bsigdep}\box0}}%
\newcommand{\ans}[1]{\hbox to #1truecm{
    \vrule height 20pt depth3.5pt width0pt
    \leaders\hrule height-2.5pt depth3pt\hfill}}

\newcommand{\ie}{\emph{i.e.{}}}

\newcommand{\half}{{\mathchoice{{\textstyle\frac12}} {{\textstyle\frac12}}
                {{\scriptscriptstyle\frac12}}{{\scriptscriptstyle\frac12}}}}
\newcommand{\one}{\mathbf{1}}
\newcommand{\bone}[1]{\one_{\{#1\}}}

\newcommand{\bet}[1]{\left [#1\right ]} 
\newcommand{\cet}[1]{\left (#1\right )} 
\newcommand{\set}[1]{\left\{#1\right\}} 
\newcount\ola \newcount\olb \newcount\olc \newcount\old \newcount\ole
\newcount\och\newcount\level
\DeclareMathOperator{\hr}{hr}

\ifx\newcolumntype\undefined
\else
  \newcolumntype{C}{>{$}c<{$}}
  \newcolumntype{L}{>{$}l<{$}}
  \newcolumntype{R}{>{$}r<{$}}
\fi
\def\OL#1{\par\noindent\hangindent=#1\parindent 
  \kern1\hangindent\ignorespaces}%
\def\ol#1{%
    \level=#1
    \ifcase\level
    \ola=0 \olb=0 \olc=0 \old=0 \ole=0\or         
    \olb=0 \olc=0 \old=0 \ole=0 \advance\ola by 1 
    \gdef\olev{\uppercase\expandafter{\romannumeral\ola}} \or
    \olc=0 \old=0 \ole=0 \advance\olb by 1        
    \och=64 \advance\och by\olb
    \gdef\olev{\char\och}\or
    \old=0 \ole=0 \advance\olc by 1               
    \och=48 \advance\och by\olc
    \gdef\olev{\char\och}\or
    \ole=0 \advance\old by 1                      
    \och=96 \advance\och by\old
    \gdef\olev{\char\och}\or
    \advance\ole by 1                             
    \gdef\olev{\romannumeral\ole} \or
    \message{Outline depth too deep: #1}\fi
    \ifnum\level>0 \OL\level\llap{\olev.\enspace}\ignorespaces\fi}%
\long\def\comment#1/*#2*/{\endcomment}%
\def\endcomment{\relax}%
\newcounter{probno}\newcounter{partno}[probno]
\newcommand{\newpart}[1][0]{\ifnum\value{partno}=0\medskip\else\vfill\fi
  \par\stepcounter{partno}\alph{partno})
  \ifnum#1<0(XC)\fi\ifnum#1=0\quad\fi\ifnum#1>0(#1)~\fi}

\makeatletter 
\multiply\count2 60 \advance\count2 \time
\edef\now{\two@digits{\the\count1}:\two@digits{\the\count2}}
\makeatother

\def\wbox#1#2#3{{\vcenter{\vbox{\hrule height.#3pt
    \hbox{\vrule width.#3pt height#1pt \kern#2pt \vrule width.#3pt}%
                            \hrule height.#3pt}}}}%
\def\Proof.{\medbreak\noindent{\bf Proof.\enspace}}
\def\qed{\nobreak{\hfill\penalty0\hbox to1truecm{}\nobreak
    \hfill$\wbox634$\par\bigskip}}%

\ifx\url\undefined
\else
\makeatletter
\def\url@rlwstyle{%
  \@ifundefined{selectfont}{\def\UrlFont{\sf}}
    {\def\UrlFont{\small\ttfamily}}}
\makeatother
\fi
\newif\ifdraft
\newif\ifsame
\newcommand{\strcfstr}[2]{%
  \samefalse
  \begingroup
    \def\1{#1}\def\2{#2}%
    \ifx\1\2\endgroup \sametrue
    \else \endgroup
    \fi}
\def\rev$Revi#1: #2 ${#2}
\def\dat$Dat#1: #2 #3 ${#2}
\def\need#1{\vskip0pt plus#1in\penalty-250\vskip0pt plus-#1in}%
\def\SVNtz$#1 -0#200#3${\global\def\tz{\ifcase#2 GMT\or-1\or-2\or ADT\or
    EDT\or EST \or CST\or MST\or PST\or AKST\or -10\or HST\else -#2\fi}}
{\ker}[1]{k(x,#1)}     

\graphicspath{{./}{Figures/New/}{Figures/}}
\begin{document}
\allowdisplaybreaks

\pagestyle{fancy}
\fancyhead[RO,LE]{\small\thepage}
\fancyhead[LO]{J. E. Johndrow and R. L. Wolpert}
\fancyhead[RE]{Tail waiting times}

\title{Model-free inference on extreme dependence via waiting times}
\author{James E. Johndrow \\ Duke University, Durham, NC USA \\
        \texttt{jj@stat.duke.edu} 
\and Robert L. Wolpert \\ Duke University, Durham, NC USA \\
         \texttt{wolpert@stat.duke.edu}
}

\maketitle
\begin{abstract}
  A variety of methods have been proposed for inference about extreme
  dependence for multivariate or spatially-indexed stochastic processes and
  time series.  Most of these proceed by first transforming data to some
  specific extreme value marginal distribution, often the unit Fr\'echet,
  then fitting a family of max-stable processes to the transformed data
  and exploring dependence within the framework of that model. The marginal
  transformation, model selection, and model fitting are all possible
  sources of misspecification in this approach.

  We propose an alternative model-free approach, based on the idea that
  substantial information on the strength of tail dependence and its
  temporal structure are encoded in the distribution of the waiting times
  between exceedances of high thresholds at different locations.  We
  propose quantifying the strength of extremal dependence and assessing
  uncertainty by using statistics based on these waiting times.  The method
  does not rely on any specific underlying model for the process, nor on
  asymptotic distribution theory. The method is illustrated
  by applications to climatological, financial, and electrophysiology data.

  To put the proposed approach within the context of the existing
  literature, we construct a class of spacetime-indexed stochastic
  processes whose waiting time distributions are available in closed form
  by endowing the support points in de Haan's spectral representation of
  max-stable processes with random birth times, velocities, and lifetimes,
  and applying Smith's model to these processes. We show that waiting times
  in this model are stochatically decreasing in mean speed, and the sample
  mean of the waiting times obeys a central limit theorem with a uniform
  convergence rate under mild conditions. This indicates that our procedure
  can be implemented in this setting using standard $t$ statistics and associated
  hypothesis tests.
\end{abstract}


\section{Introduction}

In applications where multivariate or spatial extremes are of interest,
typically one has a collection of observations
$w(\bs x,t)=(w(x_1,t), \ldots,w(x_n,t))$ of a stochastic process
$\{W(x,t)\}$ at a collection of locations $x_1,\ldots,x_n$ and times
$t_1,\ldots,t_p$ in some study period $[T_0,T_0+T]$. These observations
could represent hourly precipitation, maximum daily wind speed, or, if we
treat the spatial index set $\mc X$ as a latent coordinate in an abstract
attribute space, essentially any multivariate time series, such as daily
stock prices. Inference often focuses on the strength of dependence at
extreme quantiles for pairs of points $x_1,x_2$.

Methods for estimation and inference often fit a particular parametric or
semi-parametric model to data, then explore dependence within the context
of this model. This is typically a model of a \emph{max-stable process}.
Prior to fitting, data have usually been transformed either by taking the
maximum over time windows \citep{tawn1988bivariate, tawn1990modelling}, or
keeping only data points where $w(\bs x,t)$ exceeds a threshold \citep
{davison1990models, smith1984threshold}, then transforming to a specific
extreme value marginal distribution such as the unit Fr\'echet.  The
implicit assumption is that these ``extreme'' data are approximately
realizations from the limiting max-stable process.  In most cases the full
likelihood under the model is intractable, so pseudo- or composite-
likelihood methods are used.

Here we propose an alternative approach to inference about dependence in
the extremes of a space-time indexed stochastic process, based on waiting
times between threshold exceedances at pairs of spatial locations.  Our
procedure has two steps.  We first compute an estimate of the distribution
of waiting times between exceedances at pairs of spatial locations under
the assumption that the process is independent at those points.  We then
estimate the distance between this ``null'' distribution and the empirical
distribution of waiting times between exceedances at these two points in a
suitable metric on probability measures.  This estimate is our basic
statistic quantifying dependence.  We also propose methods for interval
estimation and measures of significance.  An advantage of this method over
fitting a specified model directly to the data is that it does not rely on
any assumptions about the underlying process and, in contrast to
alternatives, is not based on any asymptotic approximation to the
distribution of observed data.  Unlike most alternatives, our approach does
not require estimation or transformation of the marginals.  Moreover, while
the term ``spatial locations'' is a useful shorthand, we emphasize that the
index set could be abstract and the locations unobserved, a point that we
illustrate with financial applications.  Our approach is also applicable to
settings where max-stable modeling is not, such as multi-hazard risk
management quantifying overall risk due to multiple sometimes-related
hazards like hurricanes, earthquakes, volcanic eruptions, and tsunamis.

A significant portion of this paper is devoted to connecting our approach
with existing methods for inference based on max-stable processes by
constructing a particular model of a max-stable process in which the
distribution of waiting times between exceedances is tractable.  The model
we construct is similar to that of \citep{davis2013max} in that both are
generalizations of \citep{smith1990max} to space-time, and may be viewed as
a special case of the models considered in \citep {embrechts2016space}.
Our aim is not to propose a new class of models for space-time indexed
max-stable processes, but to construct a particular space-time indexed
max-stable process in which we can easily study the distributions of
waiting times.  We show that waiting times in this model are stochatically 
decreasing in mean speed, and the sample mean of the waiting times obeys a 
central limit theorem. This indicates that in this setting, a form  
of our procedure can be implemented using standard $t$ tests for the
difference of means.

A \emph {max-stable process} is defined by \citet{de1984spectral} to be a
stochastic process $Y(x)$ on an index set $\mc{X}$ with the property that,
for all integers $n\in\bbN$,
\begin{align}
 Y(\cdot) \distd \frac1n \bigvee_{i =1}^n Y_i(\cdot), \label{eq:maxstable}
\end{align}
where $\{Y_i\}$ are iid copies of $Y$, where ``$\vee$'' denotes pointwise
maximum, and where for two stochastic processes $Y,Z$ 
the relation $Z(\cdot)\distd Y(\cdot)$ means that all their
finite-dimensional marginal distributions agree.  Slightly different
definitions appear elsewhere in the literature
(\citealp {smith1990max}; 
\citealp[ \S9.3] {coles2001introduction};
\citealp{schlather2002models};
\citealp[ \S8.2] {beirlant2006statistics};
\citealp[ \S9.2] {de2006extreme}).
Some authors use the term ``\emph{simple} max-stable process'' for those
which satisfy \eqref{eq:maxstable} (and hence have Fr\'echet univariate
marginal distributions with shape $\alpha=1$) and extend the class of
max-stable processes to those satisfying $Y(\cdot) \distd [\vee_{i =1}^n
Y_i(\cdot)-b_n(\cdot)]/a_n(\cdot)$
for suitable sequences of functions $a_n(\cdot)>0$, $b_n(\cdot)$.  In the
spatial or spatio-temporal setting, one usually takes $\mc{X} = \bb{R}^d$ for
some integer $d$.

There is a large literature on parametric models for max-stable
processes. Two approaches are common for model building. The first uses the
characterization of de Haan \citep{de1984spectral, resnick1991random}: a
process $Z(x) := \sup_j u_j \ker{\xi_j}$ is max-stable if
$k : \mc X \times \mc X \to \bb R_+$ is a nonnegative kernel satisfying
$\int_\cX\,\ker{\xi}\, m(d\xi)=1$ and $\{u_j,\xi_j\}$ are points of a
Poisson random Borel measure on $\bbR_+\times\mc X$ with intensity measure
proportional to $u^{-2}\,du\,m(d\xi)$ for some $\sigma$-finite Borel
reference measure $m(dx)$ on $\cX$. \citet{smith1990max} uses Gaussian
kernels to construct a model in which the joint distribution of the process
at two points is tractable (see also \citep{coles1996modelling} and
\citep{schlather2003dependence}).  Theoretical characterizations of this
model are found in \citet{husler1989maxima}.  Another example of this
approach is the circular model of \citet{coles1994directional}.

The other main approach is based on the \emph{spectral measure} 
\citep[ \S 8.2.3]{beirlant2006statistics}.  The most popular, and oldest,
parametric model is the logistic model \citep{gumbel1960bivariate,
  gumbel1960distributions}. Extensions of this model beyond the bivariate
case were described by \citet{tawn1990modelling,
  coles1991modelling}. Another example is the Dirichlet model of
\citet{coles1991modelling}.

In either approach to max-stable modeling, data must be transformed prior
to fitting, either by taking maxima over time windows or selecting only 
data that exceed a
threshold. There are several varieties of the latter, including: keeping
observations at times $t$ for which $\max_{i=1}^n w(x_i,t)$ exceeds a
pre-specified threshold \citep{rootzen2006multivariate,
  buishand2008spatial}; fixing a specific component (say, $x_1$) and
keeping observations at times $t$ where $w(x_1,t)$ exceeds a threshold
\citep{heffernan2004conditional, heffernan2007limit, das2011conditioning,
  balkema2007high}; and, keeping all observations at times $t$ where some
vector norm $\|w(\cdot,t)\|$ exceeds a threshold \citep
{coles1991modelling, ballani2011construction}. In addition, one usually
needs to estimate the margins; see \citet[ \S 9.3]
{beirlant2006statistics}.  Fitting is usually performed using approximate
likelihood methods.

A potential shortcoming of generic max-over-windows and peaks-over-thresholds 
approaches is that some temporal information is lost
by the transformation process. More recently, a number of models explicitly 
incorporating a time dimension have been proposed. \citet{huser2014space} extend
a model of \citet{schlather2002models} to the space-time setting. The model
is constructed using the de Haan characterization, and thus the method essentially
treats time as one component of the index set $\mc X$. \citet{davis2013max} 
similarly
extend the model of \citet{smith1990max} to the space-time setting, again 
treating
one component of the index set $\mc X$ as the time domain. One restrictive
feature of these models is that the nature of dependence across space and
time is the same. More recently,  
\citet{embrechts2016space} proposes models in which temporal dependence 
structure
can be different from spatial dependence structure when the process is Markovian in time.

\section{Inference based on waiting times} \label{sec:method} 

\subsection{Basic extremal dependence measure}
Begin with a real-valued space-time indexed stochastic process
\be
Y: \mc X \times \mc T \to \bb R,
\ee
typically with spatial coordinate $x \in \mc X = \bb R^d$ and time
coordinate $t \in \mc T$ with either continuous time $\mc T = \bb R_+$
or discrete time $\mc T = \bb Z_0 = \{0,1,2,\cdots\}$ with the property
that, for each fixed $x \in \mc X$, the process
\be
t \mapsto Y(x,t)
\ee
is stationary and strong Markov. Our procedure does not require Markovianity,
but it is more intuitive to motivate it in this context.

Fix a collection of (perhaps high) thresholds $y(x), x \in \mc X$, and 
consider the $\mc X$ and $\mc X \times \mc X$-indexed processes
\be
V(x) &:= \inf\{ t > 0 : Y(x,t) > y(x) \} \label{eq:V} \\
Z(x_1,x_2) &:= \inf\{ t > V(x_1) : Y(x_2,t) > y(x_2) \} - V(x_1) \label{eq:Z}, 
\ee
the waiting time until first exceedance of $y(x)$ at $x$, and the waiting time
until the first exeedance of $y(x_2)$ at $x_2$ subsequent to an exceedance of 
$y(x_1)$ at $x_1$. Let $\d$ be a semimetric on the space of probability 
measures and define
\be \label{eq:gammad}
\gamma_\d(x_1,x_2) := \d(\mc L\{V(x_2)\},\mc L\{Z(x_1,x_2)\}),
\ee
where $\mc L(Z)$ is the law of the random variable $Z$. If 
$Y(x_2,t) \ci Y(x_1,t)$, then $\gamma_\d(x_1,x_2) = 0$, whereas if there is 
strong dependence at high quantiles between $Y(x_1,\cdot)$ and $Y(x_2,\cdot)$, 
we would expect $\gamma_\d$ to be large. Thus, we propose to use 
estimates of $\gamma_\d(x_1,x_2)$ based on samples of $Y$ at pairs of locations 
$x_1,x_2$ to quantify the strength of dependence at high quantiles. 

In many applications, 
it is likely that dependence at high quantiles of $Y$ between locations $x_1$ 
and $x_2$ would result in the expectation of $Z(x_1,x_2)$ being smaller than 
the expectation of $V(x_2)$. In other words, extreme events at $x_1$ would tend 
to be followed soon thereafter by extreme events at $x_2$. In this case, we can 
just choose the semimetric
\be \label{eq:meandiff}
\d(\mu,\nu) = M(\mu,\nu) := \left| \int z (\mu-\nu)(dz) \right|, 
\ee
the absolute difference in the expectations. We consider other, more 
general, choices of $\d$ later. 

\subsection{Procedure: estimation of $\mc L(V(x_2))$ and $\mc L(Z(x_1,x_2))$}
Suppose initially that we observe $Y(x,t)$ at all times $t \ge 0$. Fix a site
$x \in \mc X$ and collections of thresholds $\underbar y(x) < \bar y(x)$ in the 
support of 
$Y(x,t)$, and set $\bar s_0 := 0$. Then for $j \in \bb N$, set 
\be \label{eq:v}
\underbar s_j(x) &:= \inf\{t > \bar s_{j-1}(x) : Y(x,t) \le \underbar y(x) \}; \\
\bar s_j(x) &:= \inf\{t > \underbar s_j(x) : Y(x,t) \ge \bar y(x) \}; \\
v_j(x) &:= (\bar s_j(x) - \underbar s_j(x)). 
\ee
The $(\underbar s_j(x), \bar s_j(x))$ are the times of the $j$th upcrossing of 
$(\underbar 
y(x), \bar y(x))$ by $Y(x,\cdot)$ and $v_j(x)$ is its duration. 

If 
instead of fixing $\underbar y(x)$, we drew $\underbar y(x)$ from the 
marginal distribution of $Y(x,t)$ -- which by assumption of stationarity does 
not depend on $t$ -- then by the strong Markov property, $\{v_j(x)\}$ would 
be an iid sequence from exactly the distribution of interest, that of $V(x)$ 
in \eqref{eq:V}. 
We expect that $\{v_j(x)\}$ will have \emph{approximately} the same 
distribution as $V(x)$
even with a fixed, appropriately chosen $\underbar y(x)$, such as the median. 
In
fact, if $Y(x,\cdot)$ were a finitely supported, discrete-time Markov process, it follows
from Theorem 1 of \cite{peres2015mixing} that the hitting time of the median is
exactly the mixing time.

If we only observe $Y(x,t)$ on some interval $t \in [0,T]$ then only a finite number
$J \ge 0$ of upcrossings will occur. In that case $\bar s_j(x)$ is infinite for 
$j 
> J$,
so $J = \max\{j \ge 0 : \bar s_j(x) < \infty\}$. It is possible for $J$ to be 
zero, 
i.e. to have no upcrossings before time $T$. In real applications the 
marginal distribution of $Y$ will be estimated from the sample, and its 
quantiles used to determine thresholds, so this will not be a concern. 
Accordingly, we use
\be \label{eq:FV}
\widehat F_{V(x)}(v) =  J^{-1} \sum_{j=1}^J \bone{v_j(x) \le v},
\ee
the empirical distribution of $\{v_j(x)\}$, as an estimator of $\mc L(V(x))$.

Estimation of $\mc L(Z(x_1,x_2))$ is similar. Fix two locations $\{x_1,x_2\} 
\subset \mc X$ and thresholds $\underbar 
y(x_1) < \bar y(x_1)$ and $\underbar y(x_2) < \bar y(x_2)$ in the 
support of $Y(x,t)$. Define $\bar{\mc S}(x) = \{\bar s_j(x)\}$ as the set of
all first exceedance times of $\bar y(x)$ at location $x$
obtained using \eqref{eq:v}. Define $s^*_0(x_2) = 0$ and for $j \in \bb N$,
set
\be
s^*_j(x_1) &:= \inf\{ t \in \bar{\mc S}(x_1) : t > s^*_{j-1}(x_2) \}; \\
s^*_j(x_2) &:= \inf\{ t \in \bar{\mc S}(x_2) : t > s^*_j(x_1) \}; \\
z_j(x_1,x_2) &:= (s^*_j(x_2) - s^*_j(x_1)) \label{eq:z}
\ee
This generates a sequence $\{z_j(x_1,x_2)\}$ of times to an exceedance at 
$x_2$ that follow one at $x_1$. A similar algorithm will generate a sequence 
$\{ z_j(x_2,x_1) \}$ of times to an exceedance at $x_1$ that follow one at 
$x_2$. We use
\be
\widehat F_{Z(x_1,x_2)}(z) = J^{-1} \sum_{j=1}^J \bone{z_j(x_1,x_2) \le z},
\ee
the analogue of \eqref{eq:FV}, as an estimator of $\mc L\{Z(x_1,x_2)\}$. We then 
use
\be
\widehat \gamma_\d(x_1,x_2) := \gamma_\d(\widehat F_{V(x_2)}, \widehat 
F_{Z(x_1,x_2)})
\ee
as our estimator of $\gamma_\d(x_1,x_2)$. In the case where $\d$ is the semimetric 
in \eqref{eq:meandiff}, one can perform approximate 
classical tests of the hypothesis 
\be
H_0: \d = 0
\ee
using the Welch $t$ statistic \cite{welch1947generalization}, and construct 
confidence intervals for the difference in means, with the one additional 
requirement that the central limit theorem holds the sample means $J^{-1} 
\sum_j v_j(x_2)$ and $J^{-1} \sum_j z_j(x_1,x_2)$.

In practice, it is often the case that $Y(x,t)$ is well-defined for all times $t \in \bb R_+$,
but is observed only at an increasing sequence of times $\{t_i\}$. In this case, 
if we begin for some $j$ with $Y(x,\underbar s_j(x))$ in the stationary 
marginal 
distribution, and set $\bar s_j(x) = \inf\{t_i > \underbar s_j(x) : Y(x,t_i) 
\ge 
\bar 
y\}$, then $v_j(x) := (\bar s_j(x) - \underbar s_j(x))$ will always 
over-estimate 
the 
actual time-to-exceedance. In fact it could over-estimate by an arbitrarily 
large amount, since it is possible for $Y(x,t^*) \ge \bar y(x)$ for an 
unobserved time $t^* \in (\underbar s_j(x), \bar s_j(x))$ that could be 
arbitrarily 
close to 
$\underbar s_j(x)$. This problem is not peculiar to our setting, and arises 
any time a continuous-time process is sampled discretely. 
If the discrete sampling frequency is high enough for the gaps $(t_j - t_{j-1})$ to be small compared to 
the typical fluctuations of $Y(x,t)$, then the discrete 
approximation will be reasonably accurate. We will assume this is the case -- were it not 
so, it would be a serious deficiency of the sampling design that would limit 
the usefulness of the data for most inference problems.

\subsection{Alternative metrics $\d$}
In some cases, the assumption that the expectation of $Z(x_1,x_2)$ decreases as 
$Y(x_1,\cdot)$ and $Y(x_2,\cdot)$ become more highly dependent at high 
thresholds is not realistic. For example, neurons firing 
in one brain region may suppress neuronal activity in another brain region. As 
such, we consider some alternatives to the semimetric in \eqref{eq:meandiff}.

A popular measure of discrepancy between empirical 
distributions is the  Anderson-Darling two-sample statistic
\be
\AD(\wh F_V, \wh F_Z) &= \frac{J_1 J_2}{J} \int_{-\infty}^{\infty} 
\frac{\{\wh F_V(z) - \wh
F_Z(z)\}^2}{\wh H_J(x)\{1-\wh H_J(x)\}} d \wh H_J(z),  
\ee
where $\wh H_J(z)$ is the empirical distribution function of the combined 
sample, 
which consists of $J = J_1+J_2$ observations. $\AD(\cdot,\cdot)$ is also a 
proper distance between finite atomic measures. 

Another distance we use is the Kolmogorov metric
\be \label{eq:Kolmogorov}
\KS(F_V,F_Z) = \sup_x |F(x) - G(x)|
\ee
for distribution functions $F_V,F_Z$. In finite samples this is estimated from 
the 
Kolmogorov-Smirnov statistic, which is just \eqref{eq:Kolmogorov} evaluated for 
empirical distribution functions $\wh F_V, \wh F_Z$. 

A third metric used here is a kernel
metric studied in the machine learning literature \cite{gretton2012kernel, gretton2006kernel, 
smola2007hilbert, song2009hilbert, sriperumbudur2010hilbert, gretton2012kernel, 
minsker2014robust} and defined 
as follows. Let $(\bb H,\langle
\cdot,\cdot \rangle_{\bb H})$ be a reproducing kernel Hilbert space on
$\THETA$ with reproducing kernel $\ck: \THETA \times \THETA \to \bb R$ and
unit ball $\bb H_1 := \{f\in\bb H: \langle f,f \rangle_{\bb H} \le 1\}$.  For
each $\theta\in\THETA$ denote by $\ck_\theta\in\bb H$ the function
$\ck(\theta,\cdot)$, that satisfies $\langle \ck_\theta, h\rangle_{\bb H} =
h(\theta)$ for all $h\in\bb H$. The set $\cP_\ck:=\{\text{ Borel probability measures 
$\mu$ on $\THETA$
  s.t. } \int_\THETA\sqrt{\ck(\theta,\theta)}\,\mu(d\theta)<\infty \}$ can be
embedded into $\cH$ by the mapping $\mu \mapsto
\mu^\bbH:=\int_\THETA\ck_\theta\,\mu(d\theta)$.  This induces a pseudo-metric
on $\cP_\ck$ by
\begin{align}
\d_\ck (\mu_1,\mu_2)&:= \|\mu_1^\bbH-\mu_2^\bbH\|_\bbH 
                  = \sup_{ h \in \bb H_1} \left|\int_{\Theta} h(\theta) 
                   (\mu_1 - \mu_2)(d\theta) \right|\label{eq:w1k}
\end{align}
\citep[ \S 2.3]{gretton2012kernel}. The kernel $\ck$ is called 
\emph{characteristic} if \eqref{eq:w1k} is in fact a metric, i.e.
$d_{\ck}(\mu_1,\mu_2) = 0$ if and only if $\mu_1 = \mu_2$. The 
Gaussian kernel $\ck_G(\theta,\theta') = \exp(-|\theta-\theta'|^2)$
is characteristic, by the definition given in \citep[\S 2.3]{minsker2014robust},
since for any non-zero signed measure $\mu$,
\be
\int \int \ck_G(\theta,\theta') \mu(d\theta) \mu(d\theta') &= \frac{1}{2 \sqrt{\pi}} 
\int \int \left\{ \int \exp(i(\theta-\theta')z)e^{-z^2/4} dz \right\} \mu(d\theta) 
\mu(d\theta') \\
&= \frac{1}{2 \sqrt{\pi}} \int \left| \int \exp(i \theta z) \mu(d\theta) \right|^2 
e^{-z^2/4} dz > 0.
\ee
We use the estimator described in \cite{gretton2012kernel} given by
\be
\hat \d_{\ck} := \frac{1}{J_1(J_1-1)} \sum_i \sum_j \ck_G(v_i,v_j) + 
\frac{1}{J_2(J_2-1)} \sum_i \sum_j \ck_G(z_i,z_j) 
-\frac{2}{J_1 J_2} \sum_i \sum_j \ck_G(v_i,z_j).
\ee

\subsection{A first application: stock price data} \label{sec:djia}

We illustrate the method by applying it to log daily returns of the 30
securities that made up the Dow Jones Industrial average as of January 1,
2015 for the period 2000--2014.  Transformation to log returns is 
common in finance, and negative values are associated with declines in
asset prices.  We put $\bar y(x) = \wh F^{-1}_{Y(x,\cdot)}(0.1)$, where 
$\wh F^{-1}_{Y(x,\cdot)}(\alpha)$ is the empirical $\alpha$ quantile for the log return 
series 
of asset $j$.  We compute waiting
times until the observed series goes \emph{below} these thresholds, so we
are interested in dependence in extreme price \emph{decreases}, or asset
price crashes. Point estimates $\wh \gamma_{\d}$ for every asset pair
using $\AD, \KS,\ck_G,$ and $\M$ for $\d$ are shown in Figure 
\ref{fig:djia_tiles}; in the case of $\d=\M$, we show the value of the $t$ 
statistic with unequal variances. These images are asymmetric, as expected, 
since $\gamma_\d(x_1,x_2) \ne \gamma_\d(x_2,x_1)$ in general, underscoring the 
sensitivity of the method to the order in which extreme events occur.
In addition to 
estimates of $\M$, we also show statistical significance for testing $H_0 : \M=0$ 
at level $0.05$. The testing indicators are not shown for $\AD$ because most of 
the pairwise tests are significant, and they are omitted for $\KS$ because the 
presence of ties renders the $p$-values inaccurate. Here and elsewhere, 
$p$-values are adjusted to obtain False Discovery Rate (FDR) control at level 
0.05 using the procedure of Benjamini and Hochberg.

Clearly, the strength of dependence in price crashes varies considerably across 
the Dow components.  Many of the pairs exhibiting the largest values 
of $\wh \gamma_\d$--- indicating strong dependence--- are easily 
anticipated. For example, among the strongest interactions is that between 
cvx (Chevron) and xom (Exxon-Mobil), in either order, two equities whose price is mainly driven 
by a single underlying factor: global oil prices. Other pairs exhibiting strong dependence are Verizon (vz) 
and AT\&T (t) and J.P. Morgan Chase (jpm) and Goldman Sachs (gs).

\begin{figure}[h]
\centering
\begin{tabular}{cc}
$\log_{10}(1+ \wh \gamma_{\AD}(x_1,x_2))$ & $\log_{10}(1+\wh \gamma_{\KS}(x_1,x_2))$ \\
\includegraphics[width=0.4\textwidth]{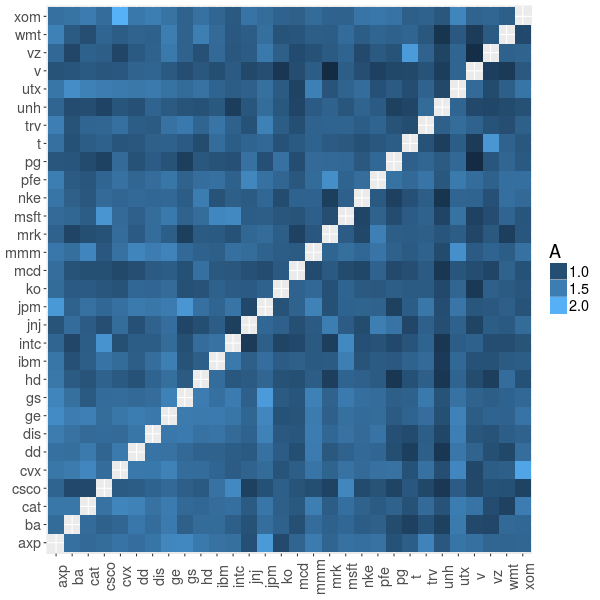} & 
\includegraphics[width=0.4\textwidth]{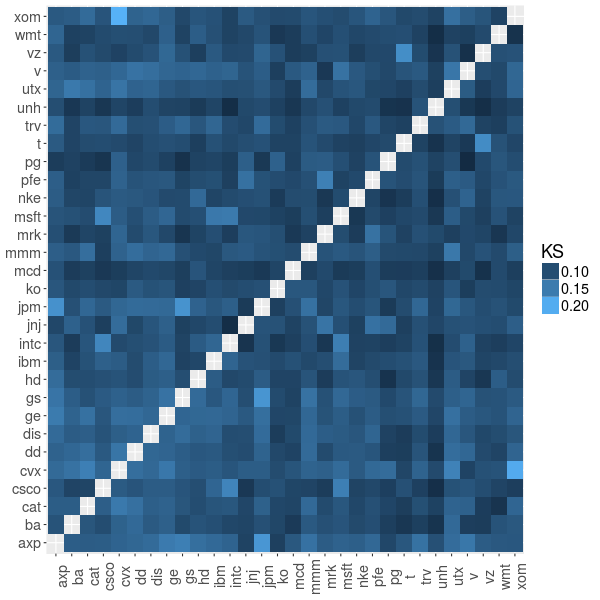} \\
$\wh \gamma_{\ck_G}(x_1,x_2)$ & $t$ statistic ($\d=\M$) \\
\includegraphics[width=0.4\textwidth]{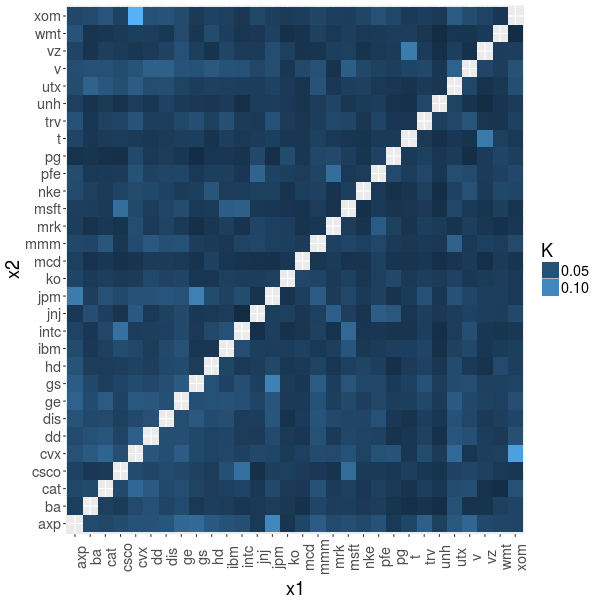} &
\includegraphics[width=0.4\textwidth]{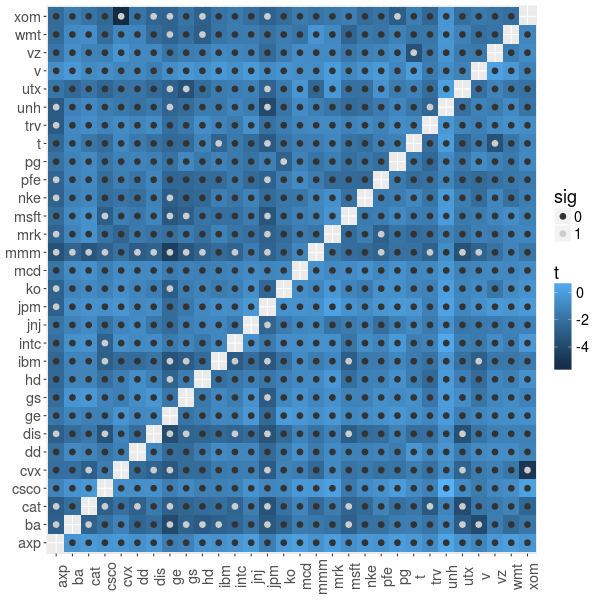} \\
\end{tabular}
\caption{Results for DJIA data. Point estimates for each of the 
different metrics are shown by shading in tile plots. $\wh \gamma_{\AD}$ and $\wh \gamma_{\KS}$ are
shown on the log scale for increased contrast. Statistical significance
is shown by overlaid points on the graphic showing $t$ statistics. }  
\label{fig:djia_tiles}
\end{figure}

Histograms of each of the four statistics across all asset pairs are shown in 
Figure \ref{fig:djia_stathists}. It is clear from these plots that while the 
$p$ values for testing $\AD=0$ and $\KS=0$ do not provide a useful way to 
identify interesting asset pairs, one can easily identify asset pairs with 
unusually strong dependence by selecting those with statistics in the right 
tail of the distribution. 

\begin{figure}
 \centering
 \includegraphics[height=0.25\textheight]{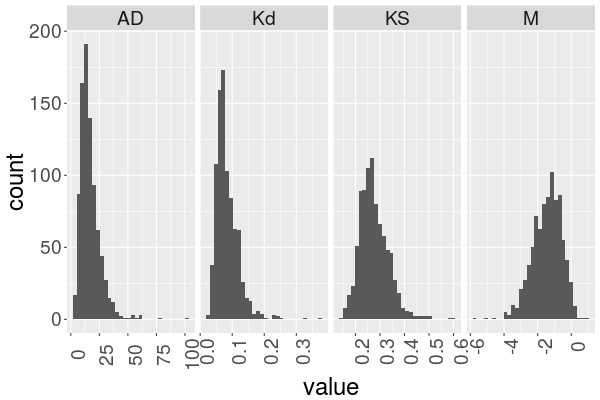} 
\includegraphics[height=0.25\textheight]{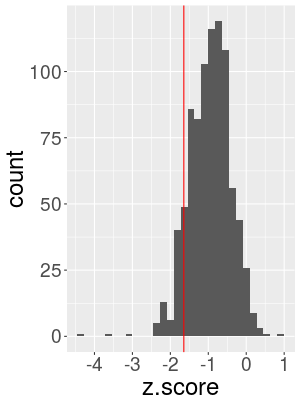}
 \caption{Left: histograms of raw statistics $\wh \gamma_{\d}$ for $\d=\AD$, 
 $\d=\ck_G$ (labeled \textnormal{Kd}), $\d=\KS$, and $\d=\M$ computed for 
every pair of assets; Right: histogram of z scores for testing $H_0: M=0$ 
after adjusting to control FDR at level 0.05.} \label{fig:djia_stathists}
\end{figure}

\subsection{Comparison to the Pickands dependence function}
A popular functional measure of extremal dependence is the bivariate
dependence function of \citet{pickands1981multivariate} (see also 
\citet[ \S 8.2.5]{beirlant2006statistics}).  For a random $n$-vector $Y$
following a max-stable distribution with distribution function $G$, define
\be \ell(v) := -\log G\{ G_1^{\leftarrow}(e^{-v_1}),\ldots,
                G_q^{\leftarrow}(e^{-v_n}) \} \ee
for $v \in \bb R_+^n$.  The function $\ell(v)$ is the \emph{stable tail
  dependence function} of $G$ (see \citep[p 257]{beirlant2006statistics}).  
  In the spatio-temporal setting, the random
variables $Y_1,\ldots,Y_n$ are associated with the process at a collection
of points, so $Y_i = Y(x_i,t)$, and repeated observations of the random
vector $Y$ correspond to sampling of the process at these locations 
at times $t_1,\ldots,t_J$. 
The \emph{Pickands
dependence function} $A(r)$ is the restriction of the bivariate tail
dependence function to the simplex
\be \label{eq:Pickands} A(r) = \ell(1-r,r), \quad r \in [0,1]. \ee
A bivariate max-stable distribution $G$ is determined by its margins
$G_1,G_2$ and $A$ by
\be G(y_1,y_2) = \exp \left[ \log\{ G_1(y_1) G_2(y_2) \} A
  \left( \frac{\log\{ G_2(y_2) \}}{\log\{ G_1(y_1) G_2(y_2) \}} \right)
  \right].  \ee
Clearly, if $A=1$, we obtain independence, while if $A$ achieves its lower
bound $A \le (1-r) \vee r$, we obtain
$G(y_1,y_2) = G_1(y_1) \wedge G_2(y_2)$, which corresponds to complete
dependence.

There is no direct representation of time in \eqref{eq:Pickands}, but if
$Y_1 = Y(x_1,t_1)$ and $Y_2 = Y(x_1,t_2)$, then $A$ is a measure of
dependence for a max-stable process at location $x_1$ at time $t_1$ and
location $x_2$ at time $t_2$.  Clearly, the existence of a pair
$\{(x_1,t_1), (x_2,t_2)\}$ for which $Y(x_1,t_1)$ is not independent of
$Y(x_2,t_2)$ is sufficient for the distribution of waiting times between
exceedances of $y_1$ at $x_1$ and $y_2$ at $x_2$ to differ from its
distribution under complete independence.  Conversely, if the distribution
of waiting times between exceedances of $y_1$ at $x_1$ and $y_2$ at $x_2$
differs from its distribution under independence, then there must exist at
least one pair $\{(x_1,t_1),(x_2,t_2)\}$ for which $A(r) \ne 1$.  While a
complete understanding of the temporal nature of dependence between
locations $x_1$ and $x_2$ depends on $A$ corresponding to every pair
$\{(x_1,t_1),(x_2,t_2)\}$, it can be fully captured by $\gamma_{\d}$ in
\eqref{eq:gammad}, at least for processes that are stationary in time.
Moreover, while $A$ only makes sense for max-stable processes, $\gamma_\d$
is meaningful for any space-time indexed stochastic process.

\section{Waiting time distributions in a model of a max-stable
  process} \label{sec:model}
In this section we construct a model of a space-time indexed max-stable
process and derive some of its properties.  Our goal is to connect the
existing literature on the statistics of multivariate extremes with the
proposed method by deriving distributions of waiting times in a model of a
max-stable process. We also show that under this model, the sample means
of waiting times satisfy a central limit theorem, and give a stochastic
dominance result for $Z$ that implies the mean waiting times are systematically
shorter under dependence.
We refer to the model as a \emph{max-stable velocity
  process}.  It is related to the space-time versions of the Gaussian
max-stable model proposed by \citet{davis2013max} and the general Markovian
max-stable models of \cite{embrechts2016space}.

\subsection{Max-stable velocity processes}
The max-stable velocity process is constructed by extending the spectral
characterization of \citet {de1984spectral} and its continuous-path
extension by \citet{resnick1991random}:
\begin{theorem}[after \citeauthor{de1984spectral} and \citeauthor
  {resnick1991random}] \label{thm:dehaan}
 Let $\{(u_j,\xi_j)\}_{j \ge 1}$ be the points of a Poisson process on
 $\bbR_+ \times \bb{R}^d$ with intensity measure proportional to $u^{-2}~du
 ~d\xi$.  Let $\{Y(x)\}_{x \in \bb{R}^d}$ be a path-continuous max-stable
 process with unit Fr\'{e}chet margins.  Then there exist 
 nonnegative continuous functions $\{\ker{\xi} : x, \xi \in \bb{R}^d\}$ such
 that
\[ \int_{\bb{R}^d} \ker{\xi} d\xi = 1 \hspace{10mm} \forall\, x \in \bb{R}^d \]
 for which
 \begin{align}
  \{Y(x)\}_{x \in \bb{R}^d} \distd
     \left\{ \sup_{j \ge 1} u_j \ker{\xi_j} \right\}_{x \in \bb{R}^d},
     \label{eq:dehaan}
 \end{align}
 where $\distd$ denotes equality in distribution.  Moreover,
 any process defined by the right side of \eqref{eq:dehaan} is max-stable.
\end{theorem}
A useful heuristic for de Haan's spectral characterization is that of weather
extremes: the locations $\{\xi_j\}$ of the support points are taken to be
storm centers, the kernel functions $\ker{\xi_j}$ describe the shape of
the storm, and the marks $\{u_j\}$ quantify storm severity.  In this context,
the process realization is the maximum over some period of time of a
climatological quantity, such as precipitation or temperature.  To create a
time-indexed process, we endow the points $\xi_j$ with birth times,
lifetimes, velocities, and shapes.  This approach has the advantage of easily
extending the physical interpretation of the points $\xi_j$ as storms or,
more generally ``events.''  Now, the storms will move and have finite
lifespans.

Specifically: fix positive numbers $\beta>0$ and $\delta>0$ and a Borel
probability measure $\PI(da)$ on a Polish ``attribute space'' $\mc A$.
Define a $\sigma$-finite Borel measure
\begin{align}
\nu(d\omega)&:=\beta u^{-2} du ~d\xi ~d\sigma 
   ~\delta e^{-\delta\tau}d\tau~\PI(da) \label{e:nu}
\end{align}
on the space $\Omega:= \bbR_+\times \bb{R}^d\times \bb{R}\times \bbR_+\times
\mc{A}$, and let $\mc{N}(d\omega) \sim \text{Po} \big(\nu(d\omega)\big)$ be a
Poisson random measure on $\Omega$ with intensity $\nu(d\omega)$--- i.e., a
measure that assigns independent random variables
$\mc{N}(B_j)\sim\text{Po}(\nu(B_j))$ to disjoint Borel sets
$B_j\subset\Omega$.  Fix a nonnegative function $ k:~ \bb{R}^d\times \bbR
\times \bb{R}^d\times\bbR \times \bbR_+ \times \mc{A}\to \bb{R}_+$ that
satisfies $\int_{\bbR^d} k(x,t; \xi,\sigma,\tau,a)\,d\xi=1$ for each
$x,t,\sigma,\tau,a$.  Using the climatological heuristic, the support points
$\{\omega_j\} = \{(u_j,\xi_j,\sigma_j,\tau_j,a_j)\}$ of $\mc{N}(d\omega)$ can
be thought of as representing storms of magnitudes $u_j>0$, initiating at
locations $\xi_j \in \bb{R}^d$ at times $\sigma_j \in\bb{R}$, with lifetimes
$\tau_j\in\bbR_+$ and attribute vectors $a_j$ that may include velocities
$v_j$, shapes $\Lambda_j$, or other features.  For any location
$x\in\bb{R}^d$ and time $t\in\bbR$, define
\begin{subequations}\label{eq:msv-both}
\begin{alignat}{2}
  Y  (x,t) &:= \sup_{j}\{ u_j\, k(x, t; \xi_j,\sigma_j,\tau_j, a_j) \} 
     \qquad &t& \in\bb{R} \label{eq:msv} \\
  Y^*(x,t) &:= \sup_{0\le s\le t} Y(x,s),&t&\ge0.\label{eq:msv-max}
\end{alignat}
\end{subequations}
We refer to \eqref{eq:msv} as a \emph{max-stable velocity (MSV) process} and 
\eqref{eq:msv-max} as the corresponding \emph{maximal MSV process}.

In the sequel we will take $\mc{A}=(\bb R^d \times \mc P^d)$ with elements
$a_j = (v_j,\Lambda_j) \in \mc{A}$ that consist of a velocity vector
$v_j\in\bbR^d$ and a shape matrix $\Lambda_j\in \cP^d$, an element of the
cone $\cP^d$ of positive-definite $d{\times}d$ matrices.  Let $\varphi: \bb
R^d \to \bb R_+$ be a continuous pdf satisfying $\int_{\bb R^d} \varphi(z) dz
= 1$ that is non-increasing in $z'z$ and, for $\Lambda \in \mc P^d$, set
$\vpL(z) := |\Lambda|^{1/2} \varphi(\Lambda^{1/2} z)$ (here $|\Lambda|$
denotes the determinant of $\Lambda\in\mc P^d$).  We take $k$ to have the
specific form
\be k(x,t;\xi,\sigma,\tau,a) 
  &= \varphi_{\Lambda}\big(x-\xi-v(t-\sigma)\big)
     \bone{\sigma\le t<\sigma+\tau}, \label{eq:phidef} \ee
the magnitude at time $t$ and location $x\in\bbR^d$ of a storm of unit
severity that originated at location $\xi\in\bbR^d$ at a time $\sigma<t$
and moved at velocity $v\in\bbR^d$ for time $(t-\sigma)$.  We write
$Y \sim \msv(\beta,\delta,\PI,\varphi)$ to denote a process of the form
\eqref{eq:msv} with $k$ defined by \eqref{eq:phidef} and $\nu(d\omega)$ as
in \eqref{e:nu}. This includes and generalizes the Gaussian and Student $t$
kernels proposed as models for the generic max-stable process by \citet
{smith1990max}.

\subsection{Main results}
Max-stable velocity processes are in fact space-time indexed forms of the max-stable process, 
as shown in Theorem \ref{thm:msvp}.  All proofs are deferred to the Appendix.
\begin{theorem}[Max-stable velocity processes are max-stable] \label{thm:msvp}
 If $Y\sim \msv(\beta,\delta,\PI,\varphi)$, then $Y$ is max-stable jointly in
 $(x,t)$ on the index set $\bb R^d \times \bb R_+$.
\end{theorem}
Theorem \ref{thm:msvp} is shown by verifying the definition in
\eqref{eq:maxstable} directly--- that is, by verifying that the process is
``stable'' under finite maxima.  It turns out that the maximal MSV process
$Y^*(x,t)$ of \eqref{eq:msv-max} is also max-stable.
\begin{theorem}[Maximal MSV process is max-stable] \label{thm:msv-max}
 If $Y\sim \msv(\beta,\delta,\PI,\varphi)$, then $Y^*(x,t)$ is max-stable in
 $(x,t)$ on the index set $\bb R^d \times \bb R_+$.
\end{theorem}
Therefore, if data are generated from a MSV process and transformed by taking
blockwise maxima, the transformed data are realizations of a max-stable
process. 

The advantage of this particular max-stable process for our purposes is 
that waiting time distributions are fairly tractable. Theorem 
\ref{thm:marginals}
gives the marginal distribution of the max-stable velocity process $Y(x,t)$
and that of the waiting times $V(x)$ until first exceedance. 
\begin{theorem}[Marginals and waiting times]\label{thm:marginals}
 The max-stable velocity process has unit Fr\'{e}chet marginals, with
 distribution function $\P [Y(x,t) \le y] = e^{-\beta/\delta y}$ for $y>0$.
 The  marginal waiting time distribution is given for $t\ge0$ by
\be \P[V(x) >t] &=
     \exp\left(- \frac{\beta}{\delta y} - t\frac{\beta}{\delta y}
     \set{\delta + \int_{\bbR^d \times \cP^d\times v^\perp} \vpL(\zeta)\,
     d\zeta \,|v|\,\PI(dv\,d\Lambda)} \right), \label{e:kap-surv}
\ee 
where $v^\perp:=\{\zeta\in\bbR^d:~v'\Lambda \zeta=0\}$ denotes the orthogonal
complement of the span of $v$ in the $\Lambda$ metric.  This is a mixture of
a point-mass at zero with weight $1-\exp(\beta/\delta y)$ and an
exponential distribution with a shape-dependent rate constant that increases
monotonically in the mean particle speed.
\end{theorem}

The next result shows that $Z(x_1,x_2)$ and $V(x)$ obey the central limit theorem,
and that the convergence rate will in most cases be uniform over all points
in the index set. This implies that one can use two-sample $t$ statistics to 
test $H_0 : \M = 0$ for our procedure when the data originate from a MSV process.
\begin{corollary} \label{cor:BerryEsseen}
The distribution of $Z(x_1,x_2)$ is stochastically dominated by a mixture of an 
exponential and an atom at zero. Therefore
\be
J^{-1} \sum_{j=1}^J z_j(x_1,x_2), \quad J^{-1} \sum_{j=1}^J v_j(x),
\ee
suitably centered and scaled, both obey the central limit theorem and converge 
to Gaussian in the Wasserstein-1 and Kolmogorov metrics. If further
\be \label{eq:UniformRate}
\sup_{x \in \mc X} \int_{\bbR^d \times \cP^d\times v^\perp} \vpL(\zeta)\,
     d\zeta \,|v|\,\PI(dv\,d\Lambda) < \infty,
\ee
then the convergence rate is uniform over all $x$.
\end{corollary}

Theorem \ref{thm:marginals} also gives immediately the distribution of $Y^*$.
\begin{corollary}[Distribution of $Y^*$] \label{cor:maxmsv}
 The distribution of $Y(x,t) = \P[Y^*(x,t) < y]$ is given by 
\eqref{e:kap-surv}.
\end{corollary}
Thus the waiting time distribution and the distribution of $Y^*(x,t)$,
unlike the marginal distribution of $Y(x,t)$, depends on velocity and
shape. For Gaussian kernels $\varphi$, \eqref{e:kap-surv} has a simple
expression:
\begin{corollary}[Waiting time distribution for Gaussian
  kernels] \label{cor:gausswaittime} For Gaussian kernels
  $\varphi(z) = (2 \pi)^{-d/2} \exp(-z'z/2)$,
\be
 \P[V(x) >t] &=
     \exp\left(- \frac{\beta}{\delta y} - t\frac{\beta}{\delta y}
     \set{\delta + \E_\PI \left[\cet{v'\Lambda v/2\pi}^{\half}\right]}\right).
     \label{eq:gausswaitmrg}
 \ee
\end{corollary}
\noindent
Thus in the Gaussian case,
the marginal waiting times converge to zero in distribution as $\E_\PI
\big[\sqrt{v' \Lambda v}\big] \to \infty$.

\subsection{Results on multivariate marginals and $Z(x_1,x_2)$}
We now give results on the joint distribution and the distribution of waiting 
times between exceedances. 
Theorem \ref{thm:joint-cdf}
gives the joint distribution of the max-stable velocity process at finite
collections of locations $\{x_i\}\iN$ and times $\{t_i\}\iN$.
\begin{theorem}[Joint CDF] \label{thm:joint-cdf}
Let $Y\sim\msv(\beta,\delta,\PI,\varphi)$ be a max-stable velocity process and
let $\{x_i\}\iN\subset\bbR^d$ and $\{t_i\}\iN\subset\bbR$ for some integer
$n\in\bbN$.  The joint CDF for $\{Y(x_i,t_i)\}\iN$ is given by
\begin{subequations}\label{e:mv-cdf}
\begin{align}
   F(y_1,\dots,y_n)&:= \P\big(\cap\iN [Y(x_i,t_i)\le y_i]\big)
                    = \exp\big(-\nu(\cup\iN B_i)\big)\label{e:mv-cdf-a}
\intertext{for $B_i:=\set{\omega:~Y(x_i,t_i)>y_i}$, with}
  \nu\big(\cup\iN B_i\big) &= \sum\iN\nu(B_i) - \sum_{i\ne j}\nu(B_i\cap B_j)
                     + \sum_{i\ne j\ne k}\nu(B_i\cap B_j\cap B_k) - \cdots,
                  \label{e:mv-cdf-b}
\intertext{the alternating sum of terms}
  \nu\big(\cap_{i\in J}B_i\big) &= \frac\beta\delta e^{-\delta(t^J-t_J)}
  \int_{\bbR^d\times \cA} \min\jJ \set{\frac{\vpL(x_j-vt_j -z)}{y_j}}
     dz\,\PI(dv\,d\Lambda)\label{e:mv-cdf-c}
\end{align}
\end{subequations}
for subsets $J\subseteq\set{1,\cdots,n}$.  Here $t^J:=\max\{t_j\}$
denotes the maximum and $t_J:=\min\{t_j\}$ the minimum of $\{t_j\}\jJ$.
\end{theorem}
While the univariate marginal distributions of $Y(x,t)$ do
not depend on $\PI(dv\, d\Lambda)$ at all, higher order marginal
distributions do.  For example,

\begin{corollary}[Gaussian MSV process joint CDF at two
  points] \label{cor:gauss-joint-dist} For the Gaussian max-stable velocity
  process with $\varphi(z)=(2\pi)^{-d/2}\exp(-z'z/2)$, the bivariate
  distribution is given for $y_1,y_2>0$ by
  $F(y_1,y_2) =\exp\big(-\nu(B_1\cup B_2)\big)$ with
\begin{align}
\nu(B_1\cup B_2) &=  \frac\beta\delta\Big(1-e^{-\delta|t_2-t_1|}\Big)
                    \Big(\frac1{y_1}+\frac1{y_2}\Big)
         +  \frac\beta\delta e^{-\delta|t_2-t_1|}  \notag\\
       \times
          \int_{\cP^d\times\bbR^d}&\set{
          \frac1{y_1} \Phi\bigg(\frac{S_\Lambda(v)}{2}
                      - \frac{\log(y_1/y_2) } {S_\Lambda(v)}\bigg) +
          \frac1{y_2} \Phi\bigg(\frac{S_\Lambda(v)}{2}
                      - \frac{\log(y_2/y_1)} {S_\Lambda(v)}\bigg) }
          \, \PI(dv\,d\Lambda) \label{eq:Fgaussmain}
\end{align}
where $\Phi(\cdot)$ is the standard Gaussian CDF and
$S_\Lambda(v):=\sqrt{\Delta_v'\Lambda\Delta_v}$ for $\Delta_v := (x_2-x_1) -
(t_2-t_1)v$.
\end{corollary}

Equation \eqref{eq:Fgaussmain} generalizes \citet [equation 3.1]
{smith1990max}, and reduces to it if $t_1=t_2$ and if $\PI$ is a unit point
mass. It leads to an explicit expression for the likelihood function for the
Gaussian max-stable velocity process with observations at two locations and times. 
This expression could be used to fit the max-stable velocity process to data
by maximum composite likelihood.

Our final result is for the distribution of
\be
Z^*(x_1,x_2) = |V(x_2) - V(x_1)|, 
\ee
which is more convenient to study than $Z(x_1,x_2)$, since we need 
not be concerned about the order in which the exceedances occur. This 
considerably simplifies the construction of sets to integrate over in obtaining 
the following result. The subsequent corollary gives a stochastic 
ordering result for either $Z(x_1,x_2)$ or $Z(x_2,x_1)$ every $x_1,x_2 \in \mc 
X$. Results are given here for the Gaussian kernel;
expressions for the general case can be found in the Appendix.

\begin{theorem}[Stochastic bound for survival function with nonzero
    velocity] \label{thm:stochbound} Suppose
  $Y\sim\msv(\beta,\delta,\PI,\varphi)$ is a Gaussian max-stable velocity
  process.  Then for any two points $x_1,x_2$ and thresholds $y_1,y_2$,
  $\jwtime$ satisfies the stochastic ordering
\be \P[Z^*(x_1,x_2) > t] &\le \exp\big\{-\nu(A) \big\}, \ee
with
\be \nu(A) =
  \frac{\beta}{\sqrt{2\pi}} \int_{\substack{\bbR^d \times \cP^d\\
            \cap \{-t \le \Delta_{12}\le t\}}}
      &(t-\Delta_{12})\,e^{-\delta \Delta_{12}}
     \bigg\{ \frac{1}{y_1} \Phi\bigg( -\frac{S^{\perp}_{\Lambda}(v)}{2} +
              \frac{\log(y_1/y_2)}{S^{\perp}_{\Lambda}(v)}
        \bigg) \\
        &+ \frac{1}{y_2} \Phi\bigg( -\frac{S^{\perp}_{\Lambda}(v)}{2}
         - \frac{\log(y_1/y_2)}{S_{\Lambda}^{\perp}(v)}
        \bigg) \bigg\} \sqrt{v'\Lambda v} \PI(d\Lambda\, dv)
        , \label{eq:jwtimeboundgauss} 
\ee
where $S^{\perp}_{\Lambda}(v) := \{(x_2-x_1)'\Lambda
 (x_2-x_1)^{\perp}\}^{1/2}$, $\Delta_{12} := v'\Lambda(x_2-x_1)/v'\Lambda v$,
 and $(x_2-x_1)^{\perp} := (x_2-x_1)-\Delta_{12} v$ is the projection of
 $(x_2-x_1)$ into the orthogonal complement of $v$ in the $\Lambda$ metric.
\end{theorem}

\begin{corollary} \label{cor:MinZ}
 At least one of
 \be
 \P[Z(x_1,x_2) > t] \le \exp\{-\nu(A)/2\}
 \ee
 or
 \be
 \P[Z(x_2,x_1) > t] \le \exp\{-\nu(A)/2\}
 \ee
 holds. Also
 \be
 \P[Z(x_1,x_2) \wedge Z(x_2,x_2) > t] \le \exp\{-\nu(A)/2\}. 
 \ee
\end{corollary}

Theorem \ref{thm:stochbound} shows that for the Gaussian MSV process, 
$Z^*(x_1,x_2)$ -- and at least one of $Z(x_1,x_2)$ and $Z(x_2,x_1)$ -- 
converges to zero in probability if
\be 
\E_\PI\left[ \sqrt{v'\Lambda v} \Phi\left(-\sqrt{(x_2-x_1)' 
    \Lambda (x_2-x_1)^{\perp}}\right) \right] \to \infty.  
\ee
This condition is only slightly stronger than the condition $\E_\PI 
[\sqrt{v'\Lambda v}] \to \infty$, which is sufficient for the marginal
waiting times to converge to zero in probability based on \eqref
{eq:gausswaitmrg}.  Informally, the difference between these conditions is
that $\PI(dv\,d\Lambda)$ cannot place too much mass on $\Lambda$ with large
eigenvalues, which corresponds to extremely concentrated kernels.  This is
intuitive: it is difficult for the same point to cause an exceedance at two
different locations when the kernels are extremely concentrated, since the
volumes of space where exceedances can occur are very small.
This result and Theorem \ref{thm:marginals} together imply that 
the waiting time distributions are stochastically decreasing in mean speed,
and the distribution of $Z(x_1,x_2)$ is dominated by a mixture
of an atom at zero and an exponential distribution. 
Moreover, the stochastic ordering also implies that at least one of 
$\E [Z(x_1,x_2)]$ and $\E [Z(x_2,x_1)]$ is decreasing in mean speed, so the 
use of $M$ for $\d$ in defining $\gamma_\d$ is in some cases reasonable.

\section{Simulation}
In this section, a simulation study is constructed to illustrate the method. We 
simulate from a MSV process and then perform inference using our 
waiting time-based procedure.
The simulation is motivated by extreme weather events, where basic scientific
knowledge allows informative choices.  Data are simulated from a Gaussian
max-stable velocity process with attribute distribution
\begin{align*}
\PI(da) = \PI(dr\, d\phi\, d\Lambda) ~\propto~ &|\Lambda|^{(\nu-k-1)/2}
e^{-\text{tr}(\Psi^{-1} \Lambda)/2} ~r^{-3/2} e^{-(a (r-m)^2)/(2m^2 r)}
\nonumber \\ ~\times~ &\prod_{h=1}^{k-1} \left[\frac{1}{2} q e^{-q
    \phi_h} \bone{\phi_h>0} + \frac{1}{2} q e^{-q \phi_h} \bone{\phi_h<0}
  \right]\,dr\,d\phi \,d\Lambda,
\end{align*}
where $(r,\phi_1,\ldots,\phi_{k-1})$ is the polar parametrization of the
velocity $v$.  This corresponds to a Wishart distribution for $\Lambda$ with
degrees of freedom $\nu$ and shape $\Psi$, an inverse Gaussian distribution
for the magnitude of the velocity $r$ with parameters $m, a$, and wrapped
Laplace distributions with parameter $q$ for the angles
$(\phi_1,\ldots,\phi_{k-1})$ defining the direction of the velocity in
$\bb{R}^k$.  The storm lifetimes $\tau \sim \Exp(\delta)$ and the support
points $\xi$ follow a homogeneous spatial Poisson process with rate $\beta
d\xi$.  The intensity measure in the specification of the process $u
~\propto~ u^{-2}$ is improper.  For the simulation, we put $u \sim
\mathrm{Pareto}(u_{\min},1)$, the conditional distribution of $u$ given $u>
u_{\min}$.  Hyperparameters for the simulation, and their units,
are shown in Table \ref{tab:hyperparameters}.

\begin{center}
\captionof{table}{Hyperparameter choices for simulations}
\begin{tabular}{ccccccccccc}
  & $\beta$ & $u_{\min}$ & $\delta$ & $\Psi$ & $\nu$ & $m$ & $a$ & $q$ \\ 
  \hline 
  value & $1/(600)$ & 1 & 1/120 & $I_2$ & 7 & 1/10 & 1/2 & 1/2 \\
  units & $(500^2 \km^2 \hr)^{-1}$ & -- & $\hr^{-1}$ & $500^2 \km^2$ & 
  -- & $500 \km \hr^{-1}$ & $500 \km \hr^{-1}$ & -- \\
\end{tabular}\label{tab:hyperparameters}
\end{center}

The data are simulated on a $10 \times 10$ box $B$.  In order to inform
hyperparameter choices, this box is taken to roughly represent a $5\,000^2
\mathrm{km}^2$ area containing the continental United States and southern
Canada.  As a result, the distributions of velocity and lifetimes of storms
are chosen to approximate the behavior of weather events in this geographic
region.  To set the time scale and allow easier interpretation of results,
one unit of time in the simulation is considered one hour.  Support points of
the marginal process $\mc{N}(d\xi\, d\sigma)$ are sampled on $B \times
[0,T]$, with $T = 114 \times 365 \times 24=998\,640$, so that the number of
support points of $\mc{N}(d\xi\, d\sigma)$ are Poisson distributed with mean
$\beta \times T \times 100 \,(500^2 \km^2 \hr)^{-1}$.  The choice of 
$\beta = 1/600 \,(500^2 \km^2 \hr)^{-1}$ gives an average
of four storms a day forming in the region.  Storm lifetimes $\tau_j$ are
sampled for each support point from $\Exp(1/120)$, which gives an average
storm lifetime of five days.  Intensities, shapes, and velocities are sampled
from the specified distributions with the hyperparameters given in Table
\ref{tab:hyperparameters}.  The values $m = 0.1, a=1/2$ for the inverse
Gaussian distribution on $r$ gives an average speed of about $50 \km \hr^{-1}$.  
The parameter $q=0.5$ places most of the mass on easterly storm
tracks.

The value of the process $Y(x,t)$ is recorded at one million homogeneously
spaced time points from $[0,T]$ at the five locations $\{x_1 = (5,5),x_2 =
(5,5.5),x_3 = (1,1),x_4 = (8,8),x_5 = (3,5)\},$ as well as twenty additional
locations sampled uniformly on $B$.  The five fixed points should result in
the process at some pairs of locations being highly tail dependent, some
pairs weakly dependent, and some nearly independent.  After simulation,
waiting times between exceedances of $y(x)=\wh F^{-1}_{Y(x,\cdot)}(0.995)$ 
are calculated at every unique pair of points. We then
estimate $\gamma_\d$ for $\d = \AD, \KS, \ck_G$, and $\M$. Plots of $\wh \gamma_\d$
against Euclidean distance between pairs of points are shown in Figure 
\ref{fig:sims}. As expected, $\wh \gamma_\d(x_1,x_2)$ is decreasing in $\|x_1 - 
x_2\|_2$. For the $t$ statistic, decreasing $\gamma_\d(x_1,x_2)$ corresponds to 
a decrease in the absolute value of the statistic, which is what we observe. 
Throughout, we show the raw $t$ statistic instead of its absolute value since 
the $t$ statistic is negative when $\E [Z(x_1,x_2)] < \E [V(x_2)]$. As 
expected based on Theorem \ref{thm:stochbound}, the sign of the $t$ statistic 
tends to be negative when $\|x_1-x_2\|_2$ is small, but is equally likely to be 
positive or negative when $\|x_1-x_2\|_2$ is large. In addition, noticeable 
differences are observed between the behavior of $\wh \gamma_\d$ with $\d=\AD$, 
$\d = \KS$, and $\d = \ck_G$ as a function of distance. In particular, when $\d 
= \KS$, the statistic is somewhat noisier as a measure of dependence.

\begin{figure}[h]
\centering
 \begin{tabular}{cc}
\includegraphics[width=0.4\textwidth]{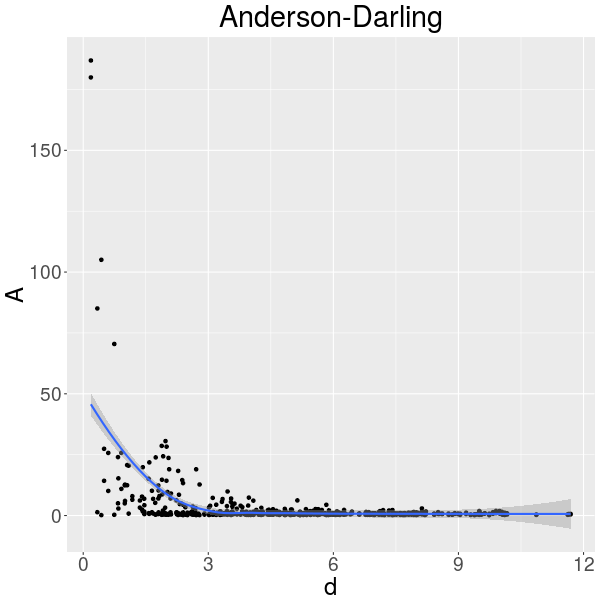} &
\includegraphics[width=0.4\textwidth]{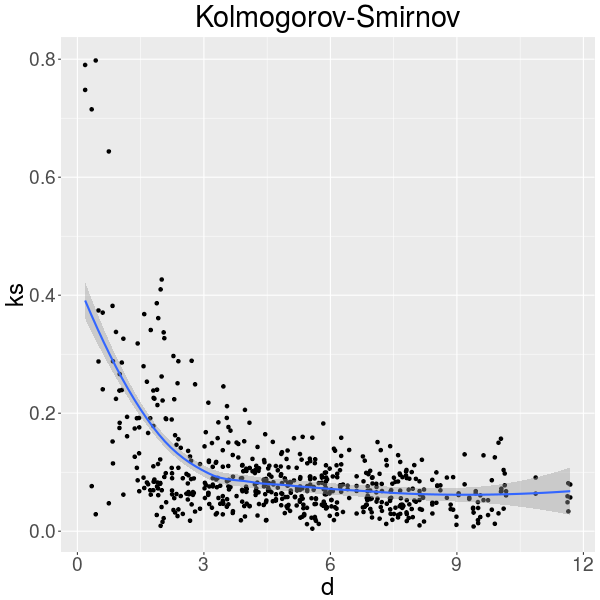} \\
\includegraphics[width=0.4\textwidth]{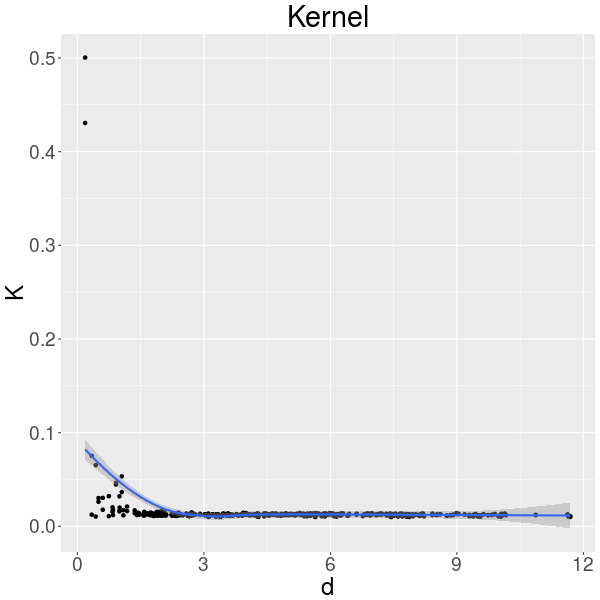} &
\includegraphics[width=0.4\textwidth]{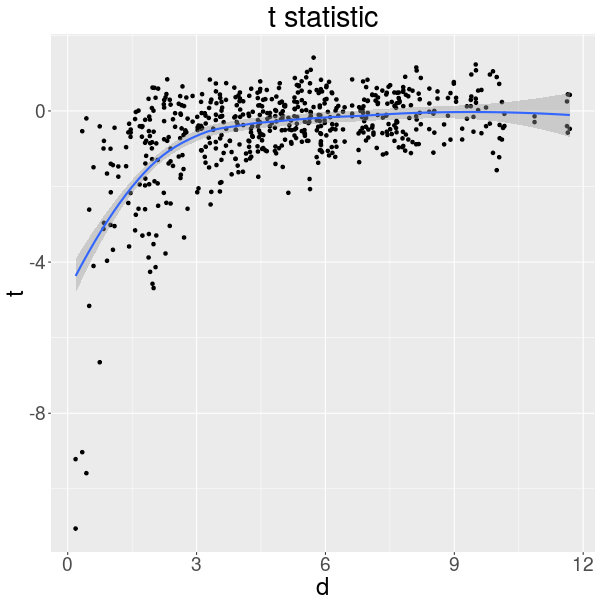} \\
\end{tabular}
\caption{$\wh \gamma_\d(x_1,x_2)$ as a function of $\|x_1 - x_2\|_2$ for 
different choices of $\d$. Results for $\ck_G$ are labeled ``\textnormal{Kernel},'' and $t$ statistics 
are shown rather than \textnormal{d=M}.} \label{fig:sims} 
\end{figure}

\FloatBarrier
\section{Applications}
The method is applied to three additional real data sets: (1) Daily 
precipitation data
for 25 weather stations in the United States for the period 1940--2014; (2)
Daily exchange rates for 12 currencies for the period 1986--1996; 
and (4) Electrical potential at 62
single neurons in the brain of a mouse exploring a maze, sampled at 500 Hz.

The first dataset is similar to the simulation study and the physical
heuristic that we introduced in Section \ref{sec:model}.  
The third--- the mouse electrophysiology data--- retains
an explicit spatial component, as the physical distance between neurons is
related to the speed at which signals can be propagated between them.  The
second application, like the Dow Jones application in Section \ref{sec:djia}, 
is financial. In these applications, the 
``spatial'' index $\mathcal X$ can be thought of as a coordinate in a latent
``attribute'' space, with the distances between assets reflecting similarity in the
factors that determine their price. No explicit reference to the index set 
$\mathcal X$ is necessary, however, to make the waiting times between threshold 
exceedances meaningful. In financial settings, these reflect the speed with 
which sentiment about particular asset classes propagates through the market.
Note that in financial
settings it is usually extremes in the \emph {left} tail that are of
particular interest, and thus it is useful to model with max-stable velocity
processes the negative of the observed data--- so exceedances of the negative
of a low threshold are the relevant events.

In choosing thresholds $\bar y(x)$ for analysis, the heuristic used is that
\be \label{eq:MinimumObs}
\min_{x_1, x_2 \in \mc X^s \times \mc X^s} \# \{z_j(x_1,x_2)\} \ge 75
\ee
where the set $\{z_j(x_1,x_2)\}$ is obtained using the procedure in 
\eqref{eq:z}, $\# A$ for a set $A$ indicates the cardinality of $A$, and $\mc 
X^s$ is the set of spatial locations at which the process is sampled.
Thresholds correspond to empirical quantiles of the observed data, and the same 
threshold is used for $\bar y(x_1), \bar y(x_2)$ in both the procedure for 
estimation of $\mc L(V(x_1))$ and $\mc L(V(x_2))$ and that for estimation of 
$\mc L(Z(x_1,x_2))$. We always use the empirical median for $\underbar y(x)$.
A consequence of
choosing thresholds in this way is that for some datasets, the empirical
quantile of the chosen threshold may be much more extreme than for others.
For example, the electrophysiology data has nearly two million observations,
so we can choose a threshold of $\bar y(x)=\wh F^{-1}_{Y(x,\cdot)}(0.995)$ 
for analysis; the exchange rate data has only about two thousand
observations, so the threshold chosen is $y_i=\wh F^{-1}_{Y(x,\cdot)}(0.90)$.  
For the Dow Jones data, extensive temporal clustering of extreme events results in the 
highest possible threshold that satisfies \eqref{eq:MinimumObs} being 
$y(x)=\wh F^{-1}_{Y(x,\cdot)}(0.90)$. For the 
precipitation data,
the threshold used is $y(x)=\wh F^{-1}_{Y(x,\cdot)}(0.975)$.
Examples
of single components of the four datasets are shown in Figure
\ref{fig:rawdat}.  In the case of the two financial datasets, the displayed
series and data used for our analysis are the log daily returns
$\log[y(x,t)/y(x,t-1)]$, following standard practice in finance.  The other
two figures show the raw data plotted in the time domain. 

\begin{figure}
\centering
\includegraphics[width=0.7\textwidth]{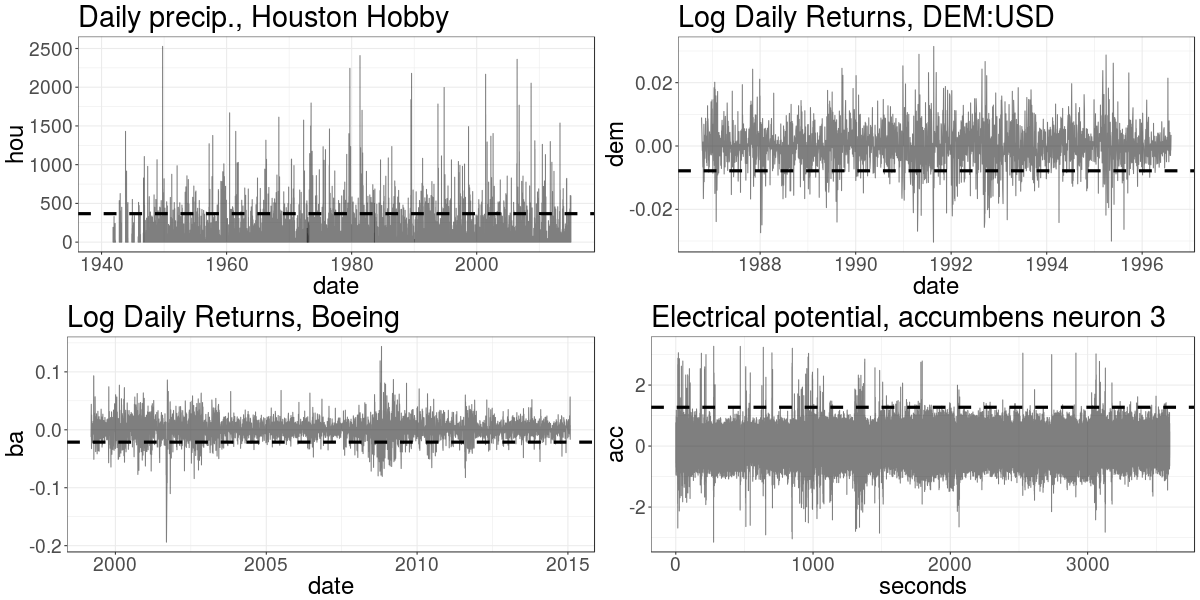} \\
\caption{Examples of raw data $y(x,t)$.  Dashed lines show $\bar y(x)$. 
The electrophysiology data are downsampled by a factor of $100$.} \label{fig:rawdat}
\end{figure}

\subsection{Precipitation} 
Results for analysis of the precipitation data are summarized in Figure
\ref{fig:weather}.  A map showing the location of each station can be found
in Figure \ref{fig:weathermap} in the Appendix. Clear geographic
structure is evident in the estimated values of $\gamdy(x_1,x_2)$.  
Overall, tail dependence is evident at nearby sites but decays with distance;
for distances greater than about 500 km the estimated values of $\gamdy$ are
all very similar. Particularly strong dependence is observed between two nearby 
cities in California: San Francisco and Sacramento. Extremely high dependence 
exists between two sites both located in Sacramento. There is noticeable 
difference between the observed dependence structure when $\KS$ is used compared 
to $\AD$. Of the $25^2-25 = 500$ estimated statistics $\wh \gamma_\d(x_1,x_2)$ 
with $\d=M$, only 21 were statistically significant. These include every combination of the 
two sites in Sacramento and the site in San Francisco in either temporal order, 
and several other sensible pairs like Baltimore and Washington, and New York and
Boston. The results are largely consistent 
with background knowledge about weather patterns in the United States and the 
relative proximity and spatial orientation of these sites. Also shown is a plot 
of $\wh \gamma_\d(x_1,x_2)$ vs $\|x_1 - x_2\|_2$ with $\d = \KS$. The expected 
pattern of decreasing dependence at increasing distance is seen.

\begin{figure}[h]
\centering
 \begin{tabular}{cc}
 \multicolumn{2}{c}{\large{$\log_{10}(1+\wh \gamma_\d(x_1,x_2))$}} \\
 $\d = \AD$ & $\d = \KS$ \\
\includegraphics[width=0.4\textwidth]{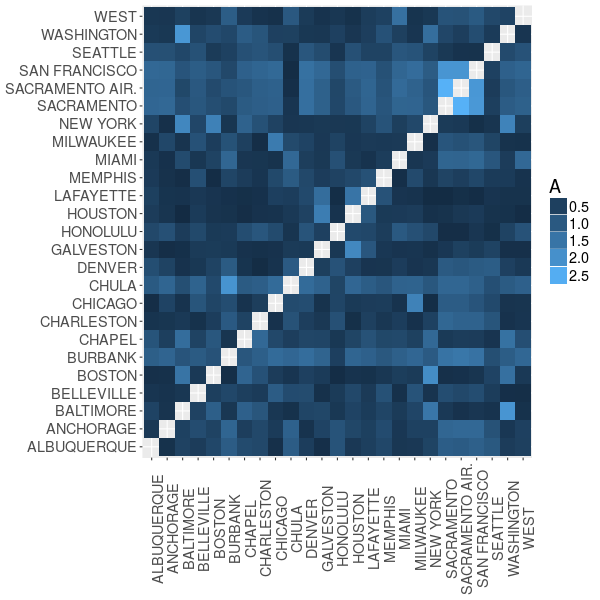} &
\includegraphics[width=0.4\textwidth]{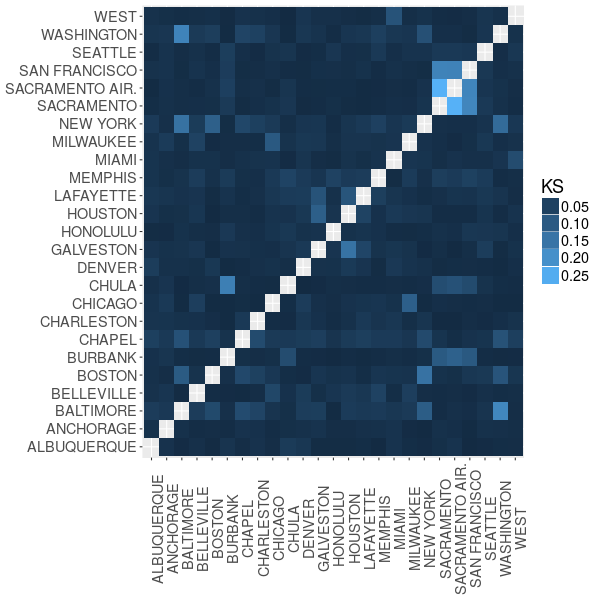} \\
$t$ statistic & $\wh \gamma_\d(x_1,x_2)$ vs $\|x_1 -x_2\|_2$ \\
\includegraphics[width=0.4\textwidth]{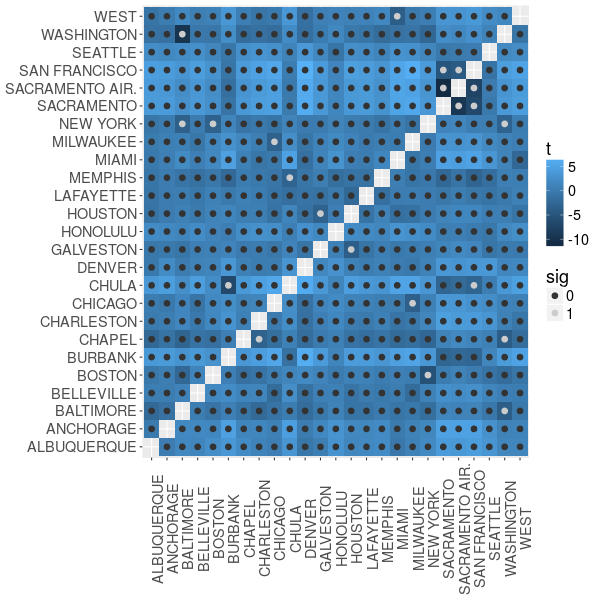} &
\includegraphics[width=0.4\textwidth]{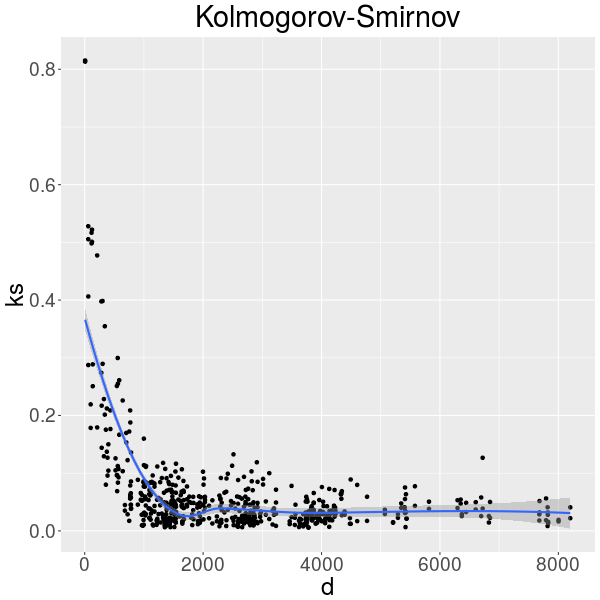} 
 \end{tabular}
\caption{Results for daily precipitation data as described in text. The values 
of $\wh \gamma_\d(x_1,x_2)$ for $\d = \KS, \AD$ in the top two panels 
are shown on the log scale for
improved contrast. Overlaid dots on $t$ statistic graphic show significance at the
0.05 level after adjusting for multiplicity (light dot: $p<0.05$).} 
\label{fig:weather}
\end{figure}

\subsection{Exchange Rates}
Exchange rate data for twelve currencies is described in \citep
{harrison1999bayesian} and \citep{prado2010time}.   
A table giving the full name of each currency and its
corresponding row/column in the color map is provided in Table
\ref{tab:currency} in the Appendix.  Figure \ref{fig:exchrate}
shows some results.  Here, there is very clear structure in the pattern of
pairwise dependence, with the European currencies (BEF, FRF, DEM, NLG, ESP,
SEK, CHF, and GBP) showing strong evidence of dependence while the other four
currencies (AUD, CAD, JPY, and NZD) show little evidence of tail dependence
among themselves or with the European currencies. Order matters as well; for example, 
SEK is highly dependent in one direction but not the other.
About one third of the 
$12^2-12 = 132$ pairs exhibit statistically significant tail dependence with 
$\d=\M$.  

\begin{figure}[h]
\centering
 \begin{tabular}{cc}
 $\log_{10}(1+\wh \gamma_\d(x_1,x_2)), \d = \AD$ & $t$ statistic \\ 
  \includegraphics[width=0.35\textwidth]{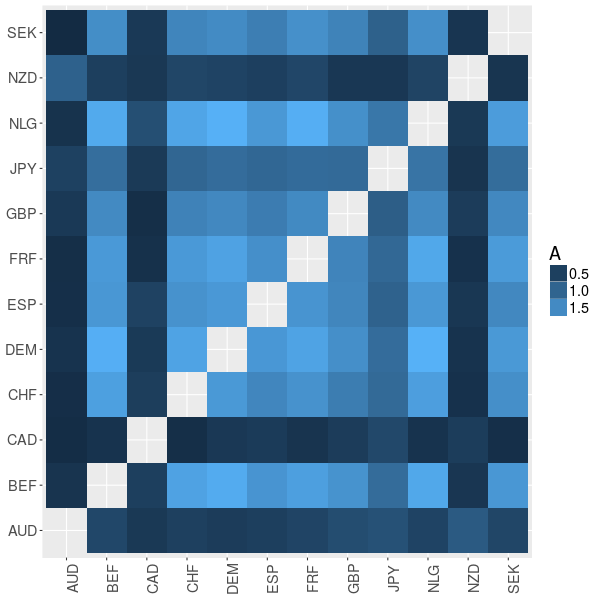} &
\includegraphics[width=0.35\textwidth]{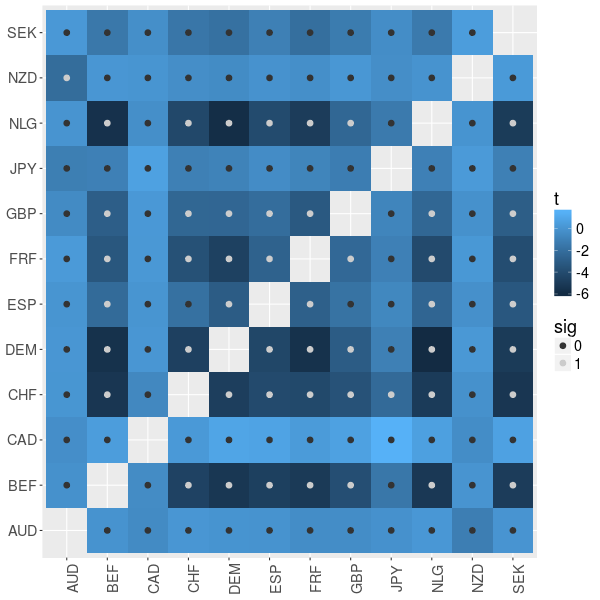} \\
 \end{tabular}
  \caption{Results for exchange rate data as described in text.} \label{fig:exchrate}
\end{figure}

\subsection{Electrophysiology}
Potential data recorded at single neurons in the brain of a mouse interacting
with a maze are described in \citep {dzirasa2010noradrenergic}.
In electrophysiology, interest lies in modeling dependence between neuron
``spikes'' at different locations in the brain.  Neuronal voltage spikes
indicate transmission of signals along axonal pathways, and large potentials
tend to cluster together in small time windows.  These events are referred to
as ``spike trains.'' Thus, in these data one expects to see extensive tail
dependence, but the waiting times between spikes at different neurons are
relevant, since they inform about the pathway that the signal takes through
the brain.

The neurons are assigned to regions of the brain, which are labeled in Figure
\ref{fig:electro_ad}.  There is ample evidence of strong
tail dependence for many pairs of neurons, but the matrix of $\wh 
\gamma_\d(x_1,x_2)$ is highly asymmetric, indicating that spikes in some 
regions tend to follow spikes in other regions. This is consistent with the 
basic understanding of how signals are propagated through the brain. For space 
reasons, we only show results for $\d=\AD$ here. 
Figure \ref{fig:electro_pval} shows a histogram of the $p$ values for testing $H_0 : 
\wh \gamma_\d(x_1,x_2) = 0$ with $\d = \M$ using the $t$ test with unequal 
variances, adjusted using the method of Benjamini and Hochberg. It is clear 
that even in this setting where dependence is very high, it is possible to 
separate more and less scientifically interesting pairs using this method. 
Performing the tests at the level of brain region instead of individual neuron 
might give more power, at the cost of spatial resolution.

\begin{figure}[h]
\centering
$\wh \gamma_\d(x_1,x_2), \d = \AD$ \\
 \includegraphics[width=.9\textwidth]{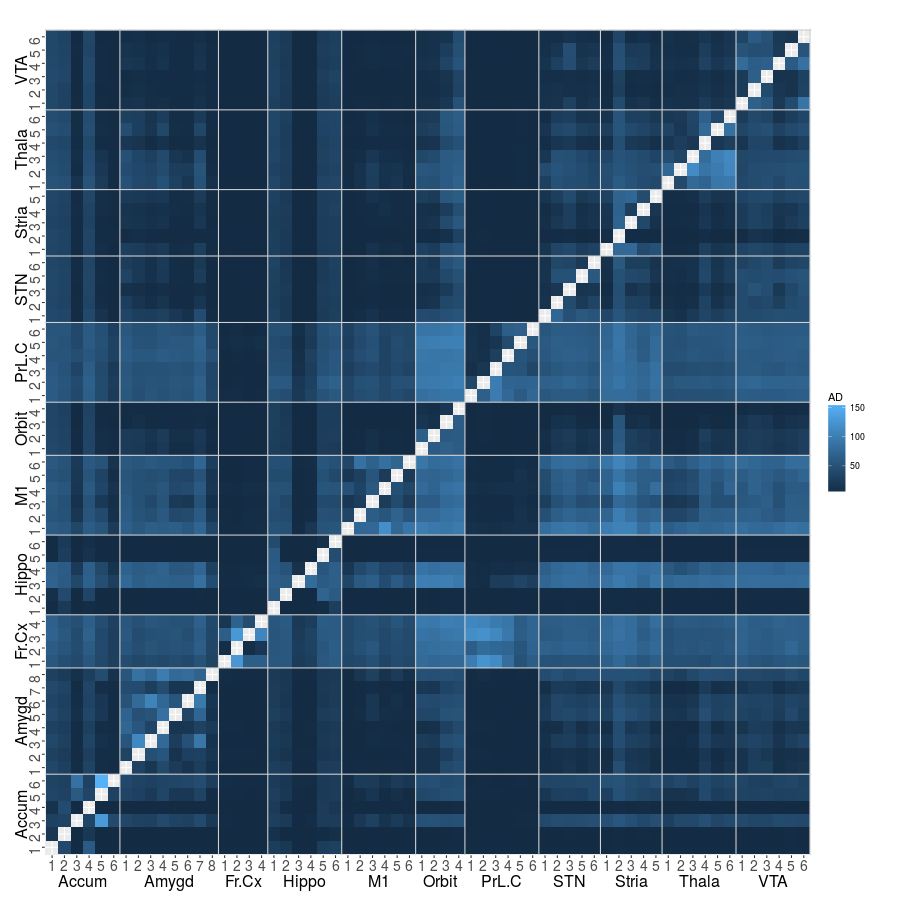} \\ 
\caption{Results for electrophysiology application as described in text. The neurons are assigned to 11 different brain regions; the axis labels show the abbreviated brain region name and an identifier for the neuron within each region. Blocks of neurons within a single brain region are delineated by horizontal and vertical white lines. } 
\label{fig:electro_ad}
\end{figure}

\begin{figure}[h]
\centering
$p$ values for $t$ tests \\
\includegraphics[width=0.4\textwidth]{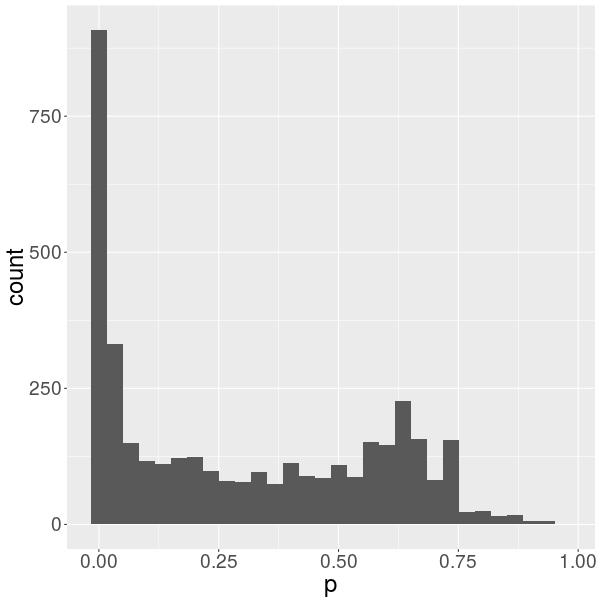}
\caption{$p$ values for $t$ tests of difference in means between $Z(x_1,x_2)$ and $V(x_2)$ for electrophysiology data.} \label{fig:electro_pval}
\end{figure}

\section{Discussion}
Characterizing tail dependence based on waiting times between peaks over
thresholds has the advantage of greater flexibility and generality than
existing alternatives in cases where temporal lags in extreme events are
possible.  The method relies strictly on the waiting times and inference on
the parameter $\gamma_\d(x_1,x_2)$, and is relatively simple and computationally 
scalable compared to fitting models of max-stable processes to data. The 
inferential method based on waiting times is robust to misspecification
of the underlying process as long as the method selected for estimation of
the waiting time distributions is sufficiently flexible. Here we have opted for 
nonparametric methods to avoid misspecification problems; where significant 
domain knowledge is available, power could be gained by using parametric 
methods to estimate $\mc L(V(x))$ and $\mc L(Z(x_1,x_2))$. 

Like other peaks-over-thresholds methods, our approach requires the choice of
appropriate thresholds.  Substantial work has been done on threshold choice
for standard peaks-over-thresholds methods.  It is unclear whether this will
translate directly to threshold choice in the waiting times context.  We have
taken the simpler approach of choosing a threshold to achieve a minimum number 
of data points.  Further work on threshold
choice is called for.  Additionally, we have thus far modeled the
pairwise waiting times entirely independently; clear gains in estimation
efficiency would result from sharing information across all pairs.  These are
worthwhile areas for future work.

\section*{Acknowledgments}
The authors thank Dan Cooley and David Dunson for helpful comments on drafts
of this manuscript, and Kristian Lum for assistance in obtaining data for the
climatological applications.  James Johndrow acknowledges funding from the
United States National Institute of Environmental Health Sciences, grant
numbers ES017436, ES020619, and ES017240, as well as the National Science 
Foundation and National Institutes of Health. Robert Wolpert acknowledges
funding from United States National Science Foundation, grant numbers
SES-1521855 and ACI-1550225.

\FloatBarrier
\begin{appendix}
\section*{Appendix}
\section{Proof of Theorems \ref{thm:msvp} and \ref{thm:msv-max}}
\subsection{Proof of Theorem \ref{thm:msvp}}
Theorem \ref{thm:msvp} is a corollary of the following result.
\begin{theorem}[Maxima of MSV Processes] \label{thm:msv-maxf}
If $Y_i\iid\msv(\beta,\delta,\PI,\varphi)$ are independent MSV processes 
for $1\le i\le n$ then their maximum
$\Ymn:=\max\iN Y_i(x,t)$ is a max-stable velocity
process with the $\msv(n\beta,\delta,\PI,\varphi)$ distribution.
\end{theorem}
\begin{proof} Each of $\{Y_i\}$ has a representation
\eqref{eq:msv} for the support points $\{\wij\}_{j\in\bbN}$ of the
$i$th of $n$ independent Poisson random measures $\Ni(d\omega)\iid
\Po\big(\nu(du)\big)$ for the intensity measure $\nu(d\omega)$ of
\eqref{e:nu}.  Then their sum $\cN^+(d\omega):= \sum\iN\Ni(d\omega)
\sim\Po\big(n\nu(du)\big)$ is also a Poisson random measure, with intensity
measure $n$ times $\nu(d\omega)$ and with support equal to the union of the
supports of $\{\Ni\}$.  The MSV processes associated with $\cN^+$ is
precisely $\Ymn$, since 
\be
\P\left[\Ymn < y \right] &= \P\left[ \cap_{i=1}^n \{Y_i < y\} \right] \\
&= \prod_{i=1}^n e^{-\nu(B)} = e^{-n \nu(B) },
\ee
for $B = \{Y_i < y\}$. But the measure $n{\cdot}\nu$ of \eqref{e:nu} is exactly
that of the $\msv(n\beta,\delta,\PI,\varphi)$ process.
\end{proof}
We now prove Theorem \ref{thm:msvp}.

\begin{proof} 
Fix $Y\sim\msv(\beta,\delta,\PI,\varphi)$ and $n$
iid max-stable copies $\{Y_i\}\iN\iid\msv(\beta,\delta,\PI,\varphi)$.  By
Theorem~\ref{thm:msv-maxf} the maximum $\Ymn$ has the $\msv (n\, \beta,
\delta, \PI,\varphi)$ distribution.  By \eqref{e:mv-cdf} in
Theorem~\ref{thm:joint-cdf}, multiply $\beta$ and each $y_j$ by $n$ to see
that all finite-dimensional marginal joint distributions of $\{\Ymn\}$ are
identical to those of $\{n\, Y(x,t)\}$, so $Y(x,t)\distd\frac1n\Ymn$
satisfies \eqref{eq:maxstable} of Theorem~\ref{thm:dehaan} and $Y$ is
max-stable.
\end{proof}

\subsection{Proof of Theorem \ref{thm:msv-max}}
We first prove two lemmas used in obtaining the main result.
\begin{lemma} \label{lem:cYcbeta}
If $Y_1\sim\msv(\beta,\delta,\PI,\varphi)$ and
$Y_c\sim\msv(c\beta,\delta,\PI,\varphi)$ for some $c>0$, then $c Y_1^* \distd
Y_c^*$.
\end{lemma}
\par\leftline{Proof.}\par
Fix $\{x_i\}\iN\subset\bbR^d$ and $\{t_i\}\iN\subset\bbR$. 
The joint CDF for $\{Y_1^*(x_i,t_i)\}$ at $\{y_i\}\subset\bbR_+$ is
\begin{align}
\P\big[\cap\iN[Y_1^*(x_i,t_i)\le y_i]\big]
   &= \exp\big(-\nu(\cup\iN A_i)\big)\label{e:joint-st}
\end{align}
for the exceedance events $A_i := [Y_1^*(x_i,t_i)> y_i]$, which may be written
\begin{align}
A_i &= \set{\omega:~ u >
   \frac{y_i} 
   {\sup\limits_{(0\vee\sigma)\le s\le (t_i\wedge\sigma+\tau)}
   \set{\vpL\big(x_i-\xi -v(s-\sigma)\big)}}}. \notag
\end{align}
By the inclusion-exclusion principle, $\nu\big(\cup A_i\big)$ can be
evaluated as
\begin{align}
\nu\big(\cup\iN A_i\big) &= \sum_{\emptyset\ne J\subset\{1,\cdots,n\}}
  (-1)^{1+|J|} \nu\big(\cap\jJ A_j\big),\label{e:iep}
\end{align}
where $|J|$ denotes the cardinality of $J$.
For a finite set $J\subset\{1,\cdots,n\}$ of indices, the intersection
\begin{align}
  \bigcap\jJ A_j 
   &= \set{\omega:~ u > \max\jJ \bet{ \frac{y_j} 
      {\sup\limits_{(0\vee\sigma)\le s\le (t_j\wedge\sigma+\tau)}
       \set{\vpL\big(x_j-\xi -v(s-\sigma)\big)}} }} \notag
\intertext{has $\nu$-measure (after changing variables from $\xi$ to
  $z:=\xi+v\sigma$)}
  \nu\cet{\bigcap\jJ A_j}
   &= \int \min\jJ  \bet{\frac
      {\sup\limits_{(0\vee\sigma)\le s\le (t_j\wedge\sigma+\tau)}
       \set{\vpL\big(x_j -vs -z\big)}}      
      {y_j}}
       dz\,\beta d\sigma\, \delta e^{-\delta\tau}d\tau\,
       \PI(dv\,d\Lambda)\label{e:capBstar}
\end{align}
The measure $\nu\big(\cap\jJ A_j\big)$ in \eqref{e:capBstar} is
unchanged if both $\beta$ and each $y_j$ are multiplied by the same constant
$c>0$.  By \eqref{e:iep}, the $\nu\big(\cup\iN A_i\big)$ is also unchanged.  

Finally, by \eqref {e:joint-st} and \eqref{e:capBstar},
\begin{align*}
\P\big[\cap\iN[c\,Y_1^*(x_i,t_i)\le y_i]\big] 
  &= \P\big[\bigcap\iN[Y_1^*(x_i,t_i)\le y_i/c]\big]\\
  &= \P\big[\bigcap\iN[Y_c^*(x_i,t_i)\le y_i]\big],
\end{align*}
proving the lemma.\qed

\begin{lemma}\label{lem:Yc}
If $\{Y_i\}\iN\iid\msv(\beta,\delta,\PI,\varphi)$ are independent MSV processes 
with maximal processes
\[Y^*_i(x,t):=\sup_{0\le s\le t} Y_i(x,s),\qquad t>0
\]
then the maximum $\Ystmn$ is the maximal process $Y^*(x,t)$ for a
max-stable velocity process $Y\sim \msv(n\beta,\delta,\PI,\varphi)$.
\end{lemma}
\par\leftline{Proof.}\par Each of $\{Y_i\}$ has a representation
\eqref{eq:msv} for the support points $\{\wij\}_{j\in\bbN}$ of the $i$th of
$n$ independent Poisson random measures $\Ni(d\omega)\iid
\Po\big(\nu(du)\big)$ for the intensity measure $\nu(d\omega)$ of
\eqref{e:nu}.  The sum $\cN^+(d\omega):= \sum\iN\Ni(d\omega)
\sim\Po\big(n\nu(du)\big)$ is also a Poisson random measure, with intensity
measure $n\,\nu(d\omega)$ and with support equal to the union $\{\wij\}$ of
the supports of $\{\Ni\}$, and the max-stable velocity process $Y(x,t)$
associated with $\cN^+(d\omega)$ by \eqref{eq:msv} is $\Ymn$.  By
Theorem~\ref{thm:msv-maxf}, $Y\sim\msv(n\beta,\delta,\pi,\varphi)$.
But $\Ystmn = Y^*(x,t)$ and the Lemma follows.\qed

We now prove Theorem \ref{thm:msv-max}.
\par\leftline{Proof.}\par
Let $\{Y_i\}\iN\iid\msv(\beta,\delta,\PI,\varphi)$ and
$Y\sim\msv(\beta,\delta,\PI,\varphi)$ be independent MSV processes 
with maximal processes
\begin{align*}
 Y^*_i(x,t)&:=\sup_{0\le s\le t} Y_i(x,s)&
   Y^*(x,t)&:=\sup_{0\le s\le t} Y (x,s),\qquad t>0
\end{align*}
By Lemma~\ref{lem:Yc}, there exists a process $Y_0\sim\msv(n\, \beta,\delta,
\PI,\varphi)$ whose maximal process $Y^*_0$ is the maximum $\Ystmn$.
By Lemma~\ref{lem:cYcbeta} with $c=n$, $Y^*_0\distd n\,Y^*$, i.e.,
$Y^*(x,t) \distd \frac1n\Ystmn$.  By \eqref {eq:maxstable} in
Theorem~\ref{thm:dehaan}, $Y^*(x,t)$ is max-stable.\qed

\section{Proof of Theorem \ref{thm:marginals} and Corollaries
  \ref{cor:gausswaittime} and \ref{cor:maxmsv}}
\subsection{Marginal Distribution of $Y(x,t)$}\label{s:Y.t-marg}
Fix $x\in\bbR^d$, $t\in\bbR$, and $y>0$.  Then the event
\[ [Y(x,t) \le y] = \big\{\cN(A)=0\}
\]
that $Y(x,t)$ does not exceed $y$ is identical to the event that
the Poisson random measure $\cN(d\omega)$ assigns zero points to the set
\[
A = \big\{\omega:~\sigma < t < \sigma+\tau,~
     u \vpL\big(x-\xi-(t-\sigma)v\big)> y \big\}
\]
of particles born before time $t$, surviving until time $t$, that move from
their birth point $\xi$ at velocity $v$ to a location
$\xi_t:=\xi+(t-\sigma)v$ sufficiently close to $x$ that their intensity $u$
will lead to exceedance of level $y$.  We will use $\nu(\cdot)$ to denote the
measure of a set with respect to the intensity measure of $\cN$.  The
probability of the event $\P(\cN(A)=0)$ for the Poisson measure
$\cN\sim\Po(\nu(d\omega))$ is
\begin{align*}
\P[ Y(x,t) \le y ] &= \exp\big(-\nu(A)\big)\\
  &= \exp\Big(-\int_{\bbR^d\times\cP^d\times\bbR^d}
               \int_{\sigma<t<\sigma+\tau}
               \int_{u>y/\vpL(x-\xi_t)}
               u^{-2}du\,\beta\delta e^{-\delta\tau}\,d\tau\,d\sigma
               \,d\xi\PI(dv\,d\Lambda)\Big)\\
  &= \exp\Big(-\int_{\bbR^d\times\cP^d\times\bbR^d}
               \int_{\sigma<t<\sigma+\tau}
               \frac1y\vpL(x-\xi_t)
               \,\beta\delta e^{-\delta\tau}\,d\tau\,d\sigma
               \,d\xi\PI(dv\,d\Lambda)\Big)\\
  &= \exp\Big(-\int_{\bbR^d\times\cP^d\times\bbR^d}
               \frac{\beta}{\delta y}\vpL\big(x-\xi-(t-\sigma)v\big)
               \,d\xi\PI(dv\,d\Lambda)\Big)\\
  &= \exp\Big(-\frac\beta{\delta y}\Big),
\end{align*}
so $Y(x,t)\sim \Fr(1,\beta/\delta)$ has the unit Fr\'echet distribution
with scale $\beta/\delta$ for all locations $x\in\bbR^d$ and times $t>0$,
for any probability distribution $\PI(dv\,d\Lambda)$.

\subsection{Marginal Distribution of $\kappa$ and $Y^*(x,t)$}\label{s:k.x-marg}
Now suppose that $\varphi(z)$ is a monotonically decreasing
function of squared Euclidean length, and denote by
$\xi_s:=\xi + v(s-\sigma)$ the location at time $s\in\bbR$ of the particle
$\omega = (u,\xi,\sigma,\tau,\Lambda,v)$, born at location $\xi$ at time
$\sigma$.  The event that the first exceedance time $\kappa(y)$ of level
$y>0$ is later than any specified $t>0$ is the event that the Poisson
random measure $\cN(d\omega)$ assigns zero points to the Borel set
\begin{align}
A &:= \Big\{\omega:~ \sup_{s\in[0,t]\cap[\sigma,\sigma+\tau]}
      u \vpL(x-\xi_s)>y\Big\},\label{e:A}
\end{align}
whose probability is $\P[\kappa(y) >t]=\exp(-\nu(A))$.  It is convenient
for us to write $A$ as the disjoint union of several simpler pieces, and
then sum their measures to find $\nu(A)$.  For fixed $\Lambda\in\cP^d$ and
$v\in\bbR^d$, denote by $\Delta$ the time interval (positive or negative)
between birth and arrival at the closest point of approach to $x$ (CPA),
starting from $\xi$ and traveling at velocity $v$, \ie,
\begin{align*}
   \Delta &:= \argmax_{s\in\bbR} \big\{\vpL(x-\xi-s v)\big\}\\
         &= \argmin_{s\in\bbR} \big\{(x-\xi-s v)'\Lambda(x-\xi-s v)\big\}\\
         &= {v'\Lambda(x-\xi)}/{v'\Lambda v}.
\end{align*}
Any vector $z\in\bbR^d$ can be written uniquely as the sum
$z=(z^\|+z^\perp)$ of its projections onto the vector space spanned by the
velocity vector $v$, and its orthogonal complement, in the inner product
$\Lambda$, given by:
\begin{align}
   z^\|   &:= (v'\Lambda z/v'\Lambda v)v&
   z^\perp &:= z-(v'\Lambda z/v'\Lambda v)v \label{e:perp+para}
\end{align}
The projection of $(x-\xi)$ onto the span of $v$ is $(x-\xi)^\|=\Delta v$,
and $\xi+\Delta v$ is the CPA to $x$.

The particle is initially \emph{approaching} $x$ if $v'\Lambda(x-\xi)>0$,
\ie, $\Delta>0$; otherwise it is initially \emph{receding} from $x$.  The
supremum in \eqref{e:A} will be attained at time $s^*=\sigma+\Delta$, if
that is in the interval $[0,t]\cap[\sigma, \sigma+\tau]$.  If not, the
supremum will be attained at one of the endpoints of that interval.

We write $A$ as the disjoint union of five sets, one for each possible time
\[s^*=\argmax\big\{Y(x,s):~s\in[0,t]\cap[\sigma,\sigma+\tau]\big\}\]
at which $Y(x,s)$ attains its maximum in $[0,t]$: the beginning $s=0$ or
end $s=t$ of the interval, the time of the particle's birth $s=\sigma$ or
death $s=\sigma+\tau$, or some intermediate point.
\begin{itemize}
\item[$A_1$:] $s^*\in(0,t)\cap(\sigma,\sigma+\tau)$.  The particle
  initially approaches $x$ (\ie, $\Delta>0$) and reaches CPA before its
  death or time $t$, whichever occurs first.  
  The supremum $\sup_s\big[\vpL(x-\xi_s)\big]$ occurs at CPA time
  $s^*=\sigma+\Delta$, at which time $\xst=\xi+\Delta v$ so
\[ (x-\xst)=x-\xi-\Delta v = (x-\xi)^\perp,
\]
the projection of $(x-\xi)$ onto the orthogonal complement 
of $v$.  It follows that 
\begin{align}
   A_1&=\set{\omega:~ u\vpL(x-\xst)>y,\qquad
     (0\vee\sigma) \le (\sigma+\Delta) \le \big(t\wedge(\sigma+\tau)\big)}
     \label{e:a1}
\intertext{After integrating wrt $u^{-2}du$,}
\nu(A_1)&=\frac1y\int_
     {\bbR^d \times \cP^d\times \bbR^d\times
      (-\Delta\le\sigma\le t-\Delta)\times
      (\Delta\le \tau<\infty)} \bone{\Delta>0}
     \beta\delta e^{-\delta\tau} \vpL(x-\xst)\,\
     d\tau d\sigma d\xi \PI(dv\,d\Lambda)\notag
\intertext{Integrating wrt $\tau$, then $\sigma$,}
     &=\frac{\beta}{y}~t~\int_
     {\bbR^d \times \cP^d\times \bbR^d} \bone{\Delta>0}
     e^{-\delta\Delta} \vpL\big((x-\xi)^\perp\big)\,\
     d\xi \PI(dv\,d\Lambda)\notag
\intertext{Changing variables from $\xi$ to $(\Delta,\zeta)$ with
     $\Delta=v'\Lambda(x-\xi)/v'\Lambda v$ (so $\Delta v=(x-\xi)^\|$) and
       $\zeta=(x-\xi)^\perp\in v^\perp\equiv\bbR^{d-1}$, with Jacobian
       $|v|=\sqrt{v'v}$, and integrating wrt $\Delta$,}
     &=\frac{\beta}{\delta y}~t~\int_
     {\bbR^d \times \cP^d\times \{\zeta\in v^\perp\}}
     \vpL(\zeta)\,d\zeta \,|v|\,\PI(dv\,d\Lambda).\label{e:nu-a1}
\end{align}

\need1
\item[$A_2$:] $s^*=t$.  The particle approaches $x$ and survives until time
  $t$ but but fails to reach CPA by $t$.  The supremum
  $\sup_s\big[\vpL(x-\xi_s)\big]$ occurs at time $s^*=t$, at which time
  $\xst=\xi+(t-\sigma) v$:
\begin{align}
   A_2&=\set{\omega:~ u\vpL\big(x-\xi-(t-\sigma)v\big)>y,\qquad
   (0\vee\sigma) \le t < \sigma + (\tau\wedge\Delta)} \label{e:a2}
\intertext{After integrating wrt $u^{-2}du$, and noting that
$(x-\xst)=(x-\xi)^\perp+(\Delta-(t-\sigma))v$ so the $d\xi$ integral
extends over the half-space on which $v'\Lambda(x-\xst)>0$,}
\nu(A_2)&=\frac1y\int_
     {\bbR^d \times \cP^d \times
      (-\infty, t] \times       
      (t-\sigma, \infty) \times 
       \bbR^d}                  
      \bone{\Delta>t-\sigma}
 \beta\delta e^{-\delta\tau} \vpL(x-\xst)\,\
     d\xi\, d\tau\, d\sigma\, \PI(dv\,d\Lambda)\notag\\
&=\frac1{2y}\int_
     {\bbR^d \times \cP^d \times
      (-\infty, t] \times       
      (t-\sigma, \infty)}       
      \beta\delta e^{-\delta\tau}\,
     d\tau\, d\sigma\, \PI(dv\,d\Lambda)\notag\\
&=\frac1{2y}\int_
    {(-\infty, t] \times        
      (t-\sigma, \infty)}       
     \beta\delta e^{-\delta\tau}\,
     d\tau\, d\sigma\notag\\
&=\frac\beta{2\delta y}.\label{e:nu-a2}
\end{align}

\need1
\item[$A_3$:] $s^*=\sigma+\tau$.  The particle lives to time $0$ and
  approaches $x$ but dies before time $t$ without reaching the CPA.
  The supremum $\sup_s\big[\vpL(x-\xi_s)\big]$ occurs at death time
  $s^*=\sigma+\tau$, at which time $\xst=\xi+\tau v$, so
\begin{align}
   A_3&=\set{\omega:~ u\vpL\big(x-\xi-\tau v\big)>y,\qquad
 (0\vee\sigma) \le \sigma + \tau \le t\wedge(\sigma+\Delta)} \label{e:a3}
\intertext{After integrating wrt $u^{-2}du$, and noting that
$(x-\xst)=(x-\xi)^\perp+(\Delta-\tau)v$ so the $d\xi$ integral
extends over the half-space on which $v'\Lambda(x-\xst)>0$,}
\nu(A_3)&=\frac1y\int_
     {\bbR^d \times \cP^d \times
     (0,\infty) \times        
     (-\tau,t-\tau) \times    
     \bbR^d}
     \bone{\Delta>\tau}
 \beta\delta e^{-\delta\tau} \vpL(x-\xi-\tau v)\,\
     d\xi\, d\sigma\, d\tau\, \PI(dv\,d\Lambda)\notag\\
    &=\frac1{2y}\int_
      {\bbR^d \times \cP^d \times
     (0,\infty) \times        
     (-\tau,t-\tau)}          
      \beta\delta e^{-\delta\tau}
       \,d\sigma\, d\tau\, \PI(dv\,d\Lambda)\notag\\
    &=\frac1{2y}\int_
    {(0,\infty) \times        
     (-\tau,t-\tau)}          
      \beta\delta e^{-\delta\tau}
       \,d\sigma\, d\tau\notag\\
    &=\frac\beta{2\delta y}~\delta t.\label{e:nu-a3}
\end{align}

\need1
\item[$A_4$:] $s^*=0$.  The particle is born before time zero, and by time
  zero is still alive and is receding from $x$ (either because it receded
  initially, or because it passed the CPA before time zero).  Because this 
  system is invariant under time-reversal, this set has the same $\nu$ 
  measure as the set $A_2$, giving $\nu(A_4) = \beta (2 \delta y)^{-1}$.


\need1
\item[$A_5$:] $s^*=\sigma$.  The particle was born during the interval
  $[0,t]$ and recedes from $x$.  Again appealing to time-reversibility,
  $\nu(A_5) = \nu(A_3) = \beta (2 \delta y)^{-1}$.
  

\end{itemize}
\subsection{Summary}\label{ss:kap-marg}
The $\nu$ measures of the five pieces are:
\begin{align}
\nu(A_1) 
  &= \frac{\beta}{\delta y}~t~\int_
     {\bbR^d \times \cP^d\times \{\zeta\in v^\perp\}}
     \vpL(\zeta)\,d\zeta \,|v|\,\PI(dv\,d\Lambda)\notag\\
\nu(A_2)  &= \frac\beta{2\delta y}\qquad
\nu(A_3)   = \frac\beta{2\delta y}~\delta t\qquad
\nu(A_4)   = \frac\beta{2\delta y}\qquad
\nu(A_5)   = \frac\beta{2\delta y}~\delta t\notag
\intertext{whose sum is}
\nu(A) &= \frac{\beta}{\delta y}\set{1+\delta t + t
     \int_{\bbR^d \times \cP^d\times \{\zeta\in v^\perp\}}
      \vpL(\zeta)\,d\zeta \,|v|\,\PI(dv\,d\Lambda)}.
      \label{e:eta.A}
\end{align}
Thus, the probability distribution of the first time $\kappa(y)$ that
$Y(x,t)$ exceeds any level $y$ is of the form
\begin{align}
  \P[\kappa(y) >t]&=\exp\big(-\nu(A)\big)
                   =\exp\Big(-\frac{\beta}{\delta y}-t\cdot \text{const}\Big),
  \label{e:kap-surv-x} 
\end{align}
a mixture of a point-mass at zero of magnitude $1-\exp(\beta/\delta y)$ and
an exponentially-distributed random variable whose rate constant
\[ \frac\beta{\delta y}\set{\delta +
     \int_{\bbR^d \times \cP^d\times v^\perp}
     \vpL(\zeta)\,d\zeta \,|v|\,\PI(dv\,d\Lambda)}
\]
depends on the mean particle speed and kernel shape. Corollary \ref{cor:maxmsv}
follows immediately since $[Y^*(x,t) < y]$ and $[\kappa(y) > t]$ are identical
events.

\subsection{Special Case: Gaussian Kernel}\label{ss:Gauss}
Consider the Gaussian case of
$\varphi(z)=(2\pi)^{-d/2}\exp(-z'z/2)$, and fix any
$\Lambda\in\cP^d$ and $v\in\bbR^d$.  The projection of $\xi\in\bbR^d$ onto
the span of $v$ (see \eqref{e:perp+para}) is
$\xi^\| = (v'\Lambda \xi/v'\Lambda v)v = \Delta v$, where
$\Delta = (v'\Lambda \xi/v'\Lambda v)$, with squared $\Lambda$-length
$(\xi^\|)'\Lambda(\xi^\|) = \Delta^2 v'\Lambda v $.  Using \eqref{e:vpL-vperp}
\begin{align}
   \int_{v^\perp}\vpL(\zeta)\,d\zeta \,|v|
  &= \cet{v'\Lambda v/2\pi}^{\half},\notag
\intertext{and $\nu(A)$ from \eqref{e:eta.A} is}
\nu(A) &= \frac{\beta}{\delta y}\set{1+\delta t + t
     \int_{\bbR^d \times \cP^d}\cet{v'\Lambda v/2\pi}^{\half}
      \,\PI(dv\,d\Lambda)}.\label{e:eta-A-Gauss}
\end{align}

\section{Proof of Corollary \ref{cor:BerryEsseen}}
It follows from Theorem \ref{thm:marginals} that $\E[|V(x_i)|^3] = c_i < \infty$ for all
$x_i \in \mc X$. Let 
\be
V^*(x_i) = \frac{J^{-1} \sum_{j=1}^J V_j(x_i) - \E[V(x_i)]}{\sqrt{n^{-1} \var(V(x_i))}},
\ee
then for an absolute constant $C$ we have
\be \label{eq:BerryEsseenV}
\sup_t |\P[V^*(x_i) < t] - \Phi(t)| \le \frac1{\sqrt{n}} C c_j,
\ee
where $\Phi(\cdot)$ is the standard normal distribution function. If \eqref{eq:UniformRate} holds, then $\sup_j c_j = c < \infty$, and we can replace $c_j$ by $c$ in \eqref{eq:BerryEsseenV} for all $i$. The same result holds in the Wasserstein 1 metric with respect to the Euclidean distance with a different constant $C$. The proof for $Z$ is essentially identical.

\section{Proof of Theorem \ref{thm:joint-cdf}}
Applying the inclusion-exclusion principle and change-of-variables from $\xi$ to
$z:=(\xi-v\sigma)$ to \eqref{e:mv-cdf-a} gives the result.

\section{Proof of Corollary \ref{cor:gauss-joint-dist}}
By Theorem \ref{thm:joint-cdf}, the joint CDF of $\Yt1$ and $\Yt2$ can be
found as 
\[ \P[\Yt1\le y_1,~\Yt2\le y_2] = \exp\big(-\nu(B_1\cup B_2)\big) \]
for the sets
\[ B_i := \big\{\omega:~\sigma < t_i < \sigma+\tau,~
     u \vpL\big(x_i-\xi-(t_i-\sigma)v\big)> y_i \big\}
\]
of particles leading to an exceedance of $y_i$ at $(x_i,t_i)$ for $i=1,2$. 
In Theorem \ref{thm:marginals}, we found $\nu(B_i)=\beta/\delta y_i$.  To find
the required $\nu\big(B_1\cup B_2)\big)=\nu(B_1)+\nu(B_2)-\nu(B_1\cap B_2)$ we
need the measure of the intersection. Putting $\xi_{t_i} = \xi+(t_i-\sigma)v$,
\begin{align}
\nu\big(B_1\cap B_2)\big) 
  &= \int_{\cP^d\times\bbR^d\times\bbR^d}
     \int_{\sigma\le(t_1\wedge t_2)}\notag\\ &\hspace*{15mm}
     \int_{\tau\ge(t_1\vee t_2)-\sigma} 
     \int_{u\ge \big(\frac{y_1}{\vpL(x_1-\xi_{t_1})}\vee\frac{y_2}{\vpL(x_2-\xi_{t_2})}\big)}
      u^{-2}du\delta e^{-\delta\tau}d\tau\,\beta d\sigma\, d\xi\,
      \PI(dv\,d\Lambda)\notag\\
  &= \int_{\cP^d\times\bbR^d\times\bbR^d}
     \int_{\sigma\le(t_1\wedge t_2)}\notag\\ &\hspace*{15mm}
     \int_{\tau\ge(t_1\vee t_2)-\sigma}
        \left(\frac{\vpL(x_1-\xi_{t_1})}{y_1} \wedge
        \frac{\vpL(x_2-\xi_{t_2})}{y_2} \right)
        \delta e^{-\delta\tau}d\tau\,\beta d\sigma\, d\xi\,
        \PI(dv\,d\Lambda)\notag\\
  &= \int_{\cP^d\times\bbR^d\times\bbR^d}
     \int_{\sigma\le(t_1\wedge t_2)}\notag\\ &\hspace*{15mm}
        \left(\frac{\vpL(x_1-\xi_{t_1})}{y_1} \wedge
        \frac{\vpL(x_2-\xi_{t_2})}{y_2} \right)
         e^{-\delta((t_1\vee t_2)-\sigma)}\,\beta d\sigma\, d\xi\,
        \PI(dv\,d\Lambda)\notag
\intertext{Set} 
\Delta_v&:=(x_2-x_1)-(t_2-t_1)v \label{eq:delta}
\intertext{(which doesn't depend on $\sigma$) and change variables from $\xi$ 
  to $z:=\xi-\frac{x_1+x_2}2 + v[\frac{t_1+t_2}2-\sigma]$:}
  &= \int_{\cP^d\times\bbR^d\times\bbR^d}
     \int_{\sigma\le(t_1\wedge t_2)}
        \left(\frac{\vpL(z-\Delta_v/2)}{y_1} \wedge
        \frac{\vpL(z+\Delta_v/2)}{y_2} \right)\notag\\ &\hspace*{75mm}
         e^{-\delta((t_1\vee t_2)-\sigma)}\,\beta d\sigma\, dz\,
        \PI(dv\,d\Lambda)\notag\\
  &= \frac\beta\delta e^{-\delta|t_2-t_1|}
       \int_{\cP^d\times\bbR^d\times\bbR^d}
           \left(\frac{\vpL(z-\Delta_v/2)}{y_1} \wedge
           \frac{\vpL(z+\Delta_v/2)}{y_2} \right)
           \, dz\, \PI(dv\,d\Lambda), \label{e:A1A1-general}
\intertext{so then}        
 \nu(A_1 \cup A_2) &=  \frac{\beta}{\delta}
     \left[ \frac{1}{y_1} + \frac{1}{y_2} \right] \notag \\
 &- \frac\beta\delta e^{-\delta|t_2-t_1|}
       \int_{\cP^d\times\bbR^d\times\bbR^d}
     \left(\frac{\vpL(z-\Delta_v/2)}{y_1} \wedge
        \frac{\vpL(z+\Delta_v/2)}{y_2} \right)
        \, dz\, \PI(dv\,d\Lambda) \notag
\intertext{In the Gaussian case with $\varphi(z)=(2\pi)^{-d/2}\exp(-z'z/2)$
  this simplifies.  Set $B^*_1:=B_1\cap B_2\cap \{\omega: \Delta'\Lambda z
  <\log\frac{y_1}{y_2}\}$, the set of $\omega\in B_1\cap B_2$ where the
  minimum in \eqref {e:A1A1-general} is attained at $(x_1,t_1)$, and set
  $B^*_2=(B_1\cap B_2)\backslash B^*_1$, the set on which it is attained at
  $(x_2,t_2)$.  Then}
\nu(B^*_1) &= \frac\beta{\delta y_1} e^{-\delta|t_2-t_1|}
       \int_{\cP^d\times\bbR^d\times\bbR^d} \vpL(z-\Delta_v/2)
       \bone{\Delta'\Lambda z < \log\frac{y_1}{y_2}}
        \, dz\, \PI(dv\,d\Lambda)\notag
\end{align}
For fixed $v$ let $\zeta:=z-(\Delta_v'\Lambda
z/\Delta_v'\Lambda\Delta_v)\Delta_v$ be the orthogonal projection onto the
space $\Delta_v^\perp$ perpendicular to $\Delta_v$ in the $\Lambda$ norm, and
change variables from $z$ to $(\zeta, s)$ with $z=\zeta+s\Delta_v$ and
Jacobian $dz = (\Delta_v'\Delta_v)^\half\,d\zeta\,ds$, to find
\begin{align}
\nu(B^*_1) 
    &= \frac\beta{\delta y_1} e^{-\delta|t_2-t_1|}
       \int_{\cP^d\times\bbR^d\times\Delta_v^\perp\times \bbR}
       |\Lambda/2\pi|^\half \exp\big(-\zeta'\Lambda\zeta/2\big)
       e^{-s^2 \Delta_v'\Lambda\Delta_v/2}\notag\\
&\hspace*{60mm} \bone{s < \log\frac{y_1}{y_2}/ \Delta'\Lambda\Delta}
       \,\sqrt{\Delta'_v\Delta_v} \, ds\,d\zeta\, \PI(dv\,d\Lambda)\notag\\
    &= \frac\beta{\delta y_1} e^{-\delta|t_2-t_1|} \int_{\cP^d\times\bbR^d}
       \Phi\Big(\frac{\log\frac{y_1}{y_2}}{S_\Lambda(v)} - \half S_\Lambda(v)\Big)
       \, \PI(dv\,d\Lambda)\notag
\intertext{where $S_\Lambda(v):=\sqrt{\Delta_v'\Lambda\Delta_v}$.
  Similarly,}
\nu(B^*_2) &= \frac\beta{\delta y_2} e^{-\delta|t_2-t_1|}    \int_{\cP^d\times\bbR^d}
       \Phi\Big(\frac{\log\frac{y_2}{y_1}}{S_\Lambda(v)} - \half S_\Lambda(v)\Big)
       \, \PI(dv\,d\Lambda),\text{\quad so}\notag\\
\nu(B_1\cap B_2) &= \nu(B^*_1)+\nu(B^*_2) = \frac\beta\delta e^{-\delta|t_2-t_1|}
          \int_{\cP^d\times\bbR^d}
          \bigg\{
           \frac1{y_1} \Phi\bigg(
                \frac{\log\frac{y_1}{y_2}}
                     {S_\Lambda(v)} - \half S_\Lambda(v)\bigg) +
                  \notag\\ &\hspace*{33mm} +
           \frac1{y_2} \Phi\bigg(
                \frac{\log\frac{y_2}{y_1}} {S_\Lambda(v)}
                       - \half S_\Lambda(v)\bigg)\bigg\}
           \, \PI(dv\,d\Lambda)\notag\\
\intertext{and}
\nu(B_1\cup B_2) &= \frac{\beta}{\delta}
          \Big(\frac1{y_1}+\frac1{y_2}\Big)
                 -\nu(B_1 \cap B_2)\notag\\
      &=  \frac\beta\delta\Big(1-e^{-\delta|t_2-t_1|}\Big)
                   \Big(\frac1{y_1}+\frac1{y_2}\Big)\notag\\
      &+  \frac\beta\delta e^{-\delta|t_2-t_1|}
          \int_{\cP^d\times\bbR^d}\bigg\{
          \frac1{y_1} \Phi\bigg(\frac12 S_\Lambda(v)
                      - \frac{\log\frac{y_1}{y_2}} {S_\Lambda(v)}\bigg) +
                 \notag\\ &\hspace*{33mm} +
          \frac1{y_2} \Phi\bigg(\frac12 S_\Lambda(v)
                      - \frac{\log\frac{y_2}{y_1}} {S_\Lambda(v)}\bigg) \bigg\}
          \, \PI(dv\,d\Lambda).\label{e:nu-Y1-Y2}
\end{align}

\section{Proof of Theorem \ref{thm:stochbound}}
In this section we compute a lower bound for the probability of two
exceedances $Y(x_j,s_j)>y_j,\quad j=1,2$ of levels $\{y_j\}\subset\bbR_+$
at locations $\{x_j\}\subset\bbR^d$ within the interval
$\{s_j\} \subset [0,t]$ for $t>0$.  Some calculations used in the sequel are found
in Section \ref{sec:integrals}.

\subsection{General case}\label{sss:cpa@1+2}
We consider the case in which CPA is attained at both points
during the interval $[0,t]$ for the general max-stable velocity process.  For $j=1,2$ set:
\begin{alignat*}2
  \Delta_j &:= \quad\argmax_{s \in \bbR} &&\set{ \vpL(x_j-\xi-sv) }
             = v'\Lambda(x_j-\xi)/v'\Lambda v
\\
      s^*_j &:= \argmax_{s \in [0,t] \cap [\sigma,\sigma+\tau]}
           && \set{ \vpL(x_j-\xi-(s-\sigma)v) }\\
      u^*_j &:= \frac{y_j}{\vpL(x_j-\xi_{s^*_j})}.
\end{alignat*}
To achieve exceedance and CPA at both locations a particle must satisfy
\begin{align}
      u &> u_1^*\vee u_2^*&
    s^*_j&=\sigma+\Delta_j \in [0,t] \cap [\sigma,\sigma+\tau],
         \quad j=1,2.\label{e:cpa2}
\end{align}
Write the set $A\subset \Omega$ of particles that satisfy these conditions
as the union of four sets
\begin{align}
   A_{ij}:=\{\omega\in \Omega:~ \Delta_i=(\Delta_1\wedge\Delta_2),\qquad 
          u^*_j = (u^*_1 \vee u^*_2)\}\quad i,j=1,2\label{e:4A}
\end{align}
characterized by which CPA occurs first and which exceedance requires the
larger mass $u$.  For example,
\[A_{11}=\set{\omega \in \Omega:~
  u > u_1^* > u_2^*,~
  (0\vee\sigma)<\sigma+\Delta_1<\sigma+\Delta_2< t\wedge(\sigma+\tau)}
\]
consists of the particles that initially approach both $x_1$ and $x_2$ and
reach CPA for both before their death.  Both suprema
$\sup_s [\vpL(x-\xi_s)]$ occur at the CPA times $s^*_j = \sigma+\Delta_j$,
at CPA locations $\xi_{s^*_j} = \xi + \Delta_jv$, so
\begin{align*}
 (x_j - \xi_{s_j^*}) = (x_j-\xi-\Delta_jv) = (x_j -\xi)^{\perp}
\end{align*}
is the projection of $(x_j-\xi)$ onto the orthogonal complement of $v$ for
$j=1,2$.  After integrating wrt $u^{-2} du$, we have
\begin{align}
  \nu(A_{11}) &= \frac{1}{y_1}
   \int_{\substack{\bbR^d
         \times \cP^d
         \times \bbR^d\notag\\
         \times (-\Delta_1 < \sigma < t-\Delta_2)\\
         \times (\Delta_2 < \tau < \infty)}}
    \bone{0<\Delta_1<\Delta_2}
    \bone{u^*_1 > u^*_2} 
    \beta \delta e^{-\delta \tau} \vpL\big((x_1 -\xi)^{\perp}\big)
           d\tau\, d\sigma\, d\xi\, \PI(dv\,d\Lambda)\notag
\intertext{Integrating with respect to $\tau$ then $\sigma$,}
    &= \frac{\beta}{y_1}
            \int_{\bbR^d \times \cP^d \times \bbR^d}
      (t-(\Delta_2-\Delta_1))\,e^{-\delta \Delta_2} \vpL((x_1-\xi)^{\perp})
    \quad\times\notag\\&\hspace*{30mm}
      \bone{0<\Delta_1<\Delta_2}\bone{\Delta_2-\Delta_1\le t}
      \bone{u^*_1 > u^*_2} d\xi\, \PI(dv\, d\Lambda).\label{e:xi-Lam-v}
\end{align}
Change variables from $\xi\in\bbR^d$ to $(\zeta,\gamma)\in v^\perp\times
v^\|$ with
\[ \zeta  :=\big(\frac{x_1+x_2}2-\xi\big)^\perp,\qquad
   \gamma :=\big(\frac{x_1+x_2}2-\xi\big)^\|
\]
and introduce the quantity
\begin{align*}
\Delta_{12} &:= (\Delta_2-\Delta_1) = v'\Lambda(x_2-x_1)/v'\Lambda v,
\end{align*}
noting that it doesn't depend on $\xi$.  With this we can write
\begin{align*}
    \Delta_1&= -\half\Delta_{12}+v'\Lambda\gamma/v'\Lambda v&
    \Delta_2&= +\half\Delta_{12}+v'\Lambda\gamma/v'\Lambda v.
\end{align*}
Introduce $\mu:=\half(x_2-x_1)^\perp$ and note that
\begin{subequations}\label{e:zeta+diff}
\begin{alignat}2
\vpL(x_1-\xi_{s^*_j})
     &= \vpL\big((x_1-\xi)^\perp\big)&
     &= \vpL(\zeta-\mu)\\
\vpL(x_2-\xi_{s^*_j})
     &= \vpL\big((x_2-\xi)^\perp\big)&
     &= \vpL(\zeta+\mu)
\end{alignat}
\end{subequations}

The limits imposed by the indicator functions in \eqref{e:xi-Lam-v} are:
\par\centerline{\begin{tabular}{CCC}
0\le \Delta_2-\Delta_1\le t&\Leftrightarrow
                     &0\le v'\Lambda(x_2-x_1)\le t v'\Lambda v\\
0\le \Delta_1        &\Leftrightarrow
                     &\Delta_{12} v'\Lambda v \le 2 v'\Lambda\gamma\\
 u_2^* \le u_1^*     &\Leftrightarrow
                     &\vpL(\zeta-\mu)/\vpL(\zeta+\mu) \le y_1/y_2
\end{tabular}}
Rewriting \eqref{e:xi-Lam-v} with this variable change, then integrating
wrt $\gamma$, gives

\begin{align}
  \nu(A_{11})
    &= \frac{\beta}{y_1}
            \int_{\bbR^d \times \cP^d \times v^\perp\times v^\|}
      (t-\Delta_{12})\,e^{-\delta \Delta_{12}/2-\delta
      v'\Lambda\gamma/v'\Lambda v} \vpL(\zeta-\mu)
    \quad\times\notag\\&\hspace*{10mm}
      \bone{0\le v'\Lambda(x_2-x_1)\le t\,v'\Lambda v}
      \bone{\Delta_{12}v'\Lambda v\le 2v'\Lambda\gamma}
      \bone{\vpL(\zeta-\mu)/\vpL(\zeta+\mu) \le y_1/y_2}\,
      d\gamma\,d\zeta\, \PI(dv\, d\Lambda)\notag\\
    &= \frac{\beta}{\delta y_1}
            \int_{\bbR^d \times \cP^d}
      (t-\Delta_{12})\,e^{-\delta \Delta_{12}}
      \bone{0\le v'\Lambda(x_2-x_1)\le t\,v'\Lambda v}
    \quad\times\notag\\&\hspace*{10mm}
\set{ \int_{\zeta\in v^\perp:~\frac{\vpL(\zeta-\mu)}{\vpL(\zeta+\mu)} \le\frac{y_1}{y_2}}
      \vpL(\zeta-\mu)\,d\zeta}
      \,|v|\, \PI(dv\, d\Lambda).\label{e:gam-zet-Lam-v}
\end{align}

Now, we find the measure of the other three sets.  First
\[A_{21}=\set{\omega\in\Omega:~
  u > u_2^* > u_1^*,~
  (0\vee\sigma)<\sigma+\Delta_1<\sigma+\Delta_2< t\wedge(\sigma+\tau)},
\]
giving 
\begin{align}
  \nu(A_{21})
    &= \frac{\beta}{\delta y_2}
            \int_{\bbR^d \times \cP^d}
      (t-\Delta_{12})\,e^{-\delta \Delta_{12}}
      \bone{0\le v'\Lambda(x_2-x_1)\le t\,v'\Lambda v}
    \quad\times\notag\\&\hspace*{10mm}
\set{ \int_{\zeta\in v^\perp:~\frac{\vpL(\zeta-\mu)}{\vpL(\zeta+\mu)} \ge\frac{y_1}{y_2}}
      \vpL(\zeta+\mu)\,d\zeta}
      \,|v|\, \PI(dv\, d\Lambda).
\end{align}
Finally, it is clear that with
\be
A_{12} &=\set{\omega \in \Omega:~
  u > u_1^* > u_2^*,~
  (0\vee\sigma)<\sigma+\Delta_2<\sigma+\Delta_1< t\wedge(\sigma+\tau)} \\
A_{22} &=\set{\omega \in\Omega:~
  u > u_2^* > u_1^*,~
  (0\vee\sigma)<\sigma+\Delta_2<\sigma+\Delta_1< t\wedge(\sigma+\tau)}  
\ee
we have
\be
  \nu(A_{12})
    &= \frac{\beta}{\delta y_1}
            \int_{\bbR^d \times \cP^d}
      (t+\Delta_{12})\,e^{\delta \Delta_{12}}
      \bone{0\le -v'\Lambda(x_2-x_1)\le t\,v'\Lambda v}
    \quad\times\notag\\&\hspace*{10mm}
\set{ \int_{\zeta\in v^\perp:~\frac{\vpL(\zeta-\mu)}{\vpL(\zeta+\mu)}
    \le\frac{y_1}{y_2}}
      \vpL(\zeta-\mu)\,d\zeta}
      \,|v|\, \PI(dv\, d\Lambda).  \\
\nu(A_{22})
    &= \frac{\beta}{\delta y_2}
            \int_{\bbR^d \times \cP^d}
      (t+\Delta_{12})\,e^{\delta \Delta_{12}}
      \bone{0\le -v'\Lambda(x_2-x_1)\le t\,v'\Lambda v}
    \quad\times\notag\\&\hspace*{10mm}
\set{ \int_{\zeta\in v^\perp:~\frac{\vpL(\zeta-\mu)}{\vpL(\zeta+\mu)}
    \ge\frac{y_1}{y_2}}
      \vpL(\zeta+\mu)\,d\zeta}
      \,|v|\, \PI(dv\, d\Lambda).
\ee

\subsection{Gaussian case}
Now take $\varphi(z)=(2\pi)^{-d/2}\exp(-z'z/2)$ and fix
$v,\Lambda\in\bbR^d\times\cP^d$.  The set of $\zeta\in v^\perp$ over which
the bracketed integral in \eqref{e:gam-zet-Lam-v} is taken can be written
as:
\begin{align*}
 \vpL(\zeta-\mu) &\le \frac{y_1}{y_2} \vpL\big(\zeta+\mu)\\
  - \half (\zeta-\mu)'\Lambda(\zeta-\mu) &\le \log\frac{y_1}{y_2} 
  - \half (\zeta+\mu)'\Lambda(\zeta+\mu)\\
 \mu' \Lambda \zeta  &\le \frac{1}{2} \log\frac{y_1}{y_2}
\end{align*}
Writing $\zeta\in v^{\perp}$ as the sum $\zeta=\zeta_\perp+\zeta_\|$ of
components orthogonal and parallel to $(x_2-x_1)^\perp$ (in the $\Lambda$
metric), by \eqref{eq:gaussperp},
\be
  \nu(A_{11})
    &= \frac{\beta}{y_1}
            \int_{\substack{\bbR^d \times \cP^d\\
            0\le \Delta_{12}\le t}}
      (t-\Delta_{12})\,e^{-\delta \Delta_{12}}
      \sqrt{v'\Lambda v/2\pi} 
     \Phi\Big(- \frac{S_{\Lambda}^{\perp}(v)}{2} +
     \frac{\log(y_1/y_2)}{S_{\Lambda}^{\perp}(v)}  
        \Big)
    \, \PI(dv\, d\Lambda).
\ee
with $S^{\perp}_{\Lambda}(v) := \{(x_2-x_1)' \Lambda (x_2-x_1)^{\perp}
\}^{1/2}$.  The other three sets have measure 
\be
  \nu(A_{21})
    &= \frac{\beta}{y_2}
            \int_{\substack{\bbR^d \times \cP^d\\
            0\le \Delta_{12}\le t}}
      (t-\Delta_{12})\,e^{-\delta \Delta_{12}}
      \sqrt{v'\Lambda v/2\pi} 
     \Phi\Big( -\frac{S^{\perp}_{\Lambda}(v)}{2} -
     \frac{\log(y_1/y_2)}{S^{\perp}_{\Lambda}(v)}  
        \Big)
    \, \PI(dv\, d\Lambda).  \\
  \nu(A_{12})
    &= \frac{\beta}{y_1}
            \int_{\substack{\bbR^d \times \cP^d\\
            0\le -\Delta_{12}\le t}}
      (t+\Delta_{12})\,e^{\delta \Delta_{12}}
      \sqrt{v'\Lambda v/2\pi} 
     \Phi\Big( -\frac{S^{\perp}_{\Lambda}(v)}{2} +
     \frac{\log(y_1/y_2)}{S^{\perp}_{\Lambda}(v)} 
        \Big)
    \, \PI(dv\, d\Lambda).  \\    
 \nu(A_{22})
    &= \frac{\beta}{y_2}
            \int_{\substack{\bbR^d \times \cP^d\\
            0\le -\Delta_{12}\le t}}
      (t+\Delta_{12})\,e^{\delta \Delta_{12}}
      \sqrt{v'\Lambda v/2\pi} 
     \Phi\Big( -\frac{S^{\perp}_{\Lambda}(v)}{2} -
     \frac{\log(y_1/y_2)}{S^{\perp}_{\Lambda}(v)} 
        \Big)
    \, \PI(dv\, d\Lambda).  
\ee
Recognizing a simple change of variables, we have
\be
\nu(A) =  \frac{\beta}{\sqrt{2\pi}} \int_{\substack{\bbR^d \times \cP^d\\
            \cap \{-t \le \Delta_{12}\le t\}}}
      &(t-\Delta_{12})\,e^{-\delta \Delta_{12}}
     \bigg\{ \frac{1}{y_1} \Phi\bigg( -\frac{S^{\perp}_{\Lambda}(v)}{2}
     +\frac{\log(y_1/y_2)}{S^{\perp}_{\Lambda}(v)} 
        \bigg) \\
    &+ \frac{1}{y_2} \Phi\bigg( -\frac{S^{\perp}_{\Lambda}(v)}{2}
        -\frac{\log(y_1/y_2)}{S^{\perp}_{\Lambda}(v)} 
        \bigg) \bigg\} \sqrt{v'\Lambda v} \PI(dv\, d\Lambda) 
\ee
where $A = A_{11} \cup A_{12} \cup A_{21} \cup A_{22}$.  So
\begin{align*}
  \P[|\kappa_2-\kappa_1| > t] 
            &\le 1-\P[\kappa_1\vee\kappa_2\le t] \\ &
           \le 1-\P[\text{ CPA Exceedances at $x_1$, $x_2$ in $[0,t]$ }]\\
            &\le \exp\big\{-\nu(A)  \big\},
\end{align*}
so the probability $\P[|\kappa_2-\kappa_1| > t] \to 0$ as
$\E_\PI\left[ \sqrt{v'\Lambda v}
  \Phi\left(-\sqrt{(x_2-x_1)' \Lambda (x_2-x_1)^{\perp}}\right) \right] \to
\infty$.

\section{Proof of Corollary \ref{cor:MinZ}}
Any point $\omega \in A$ for the set $A$ in Theorem \ref{thm:stochbound} is in 
either
\be
B_1 := \{ \omega : Z(x_1,x_2) < t \}, \text{ or }
B_2 := \{ \omega : Z(x_2,x_1) < t \}. 
\ee
Letting $A_1 = A \cap B_1$, $A_2 = A \cap B_2$, we have $A = A_1 \cup A_2$, and 
therefore 
\be
\nu(A) \le \nu(A_1) + \nu(A_2),
\ee
so either $\nu(A_1) > \frac12 \nu(A)$ or $\nu(A_2) > \frac12 \nu(A)$. 

\section{Some Important Integrals} \label{sec:integrals}
\par\bigskip
As before, fix $v,x_1,x_2,\xi\in\bbR^d$ and $\Lambda\in\cP^d$, and set
\[\Delta_j:=v'\Lambda(x_i-\xi)\qquad
  \Delta_{12}:=(\Delta_2-\Delta_1)=v'\Lambda(x_2-x_1)\qquad
  \mu := \half(x_2-x_1)^\perp
\]
where, as before, for any $z\in\bbR^d$ we denote the projections of $z$
parallel and orthogonal to $v$ in the $\Lambda$ metric by
\[ z^\|:=(v'\Lambda z/v'\Lambda v)v \qquad z^\perp = z-z^\|
\]
and the value of the kernel function at $z$ by
\[ \vpL(z):=|\Lambda/2\pi|^\half\,\exp\big(-z'\Lambda z/2\big)
\]

\par\bigskip Then
\begin{align}
\int_{\bbR^d}\exp\big(-\half\xi'\Lambda\xi\big)\,d\xi 
  &= |\Lambda/2\pi|^\half\label{e:Rd}\\
\int_{v^\|}  \exp\big(-\half\gamma'\Lambda\gamma\big)\,d\gamma
  &=\int_{\bbR} \exp\big(-s^2v'\Lambda v/2\big)\,|v|\,ds\notag
\intertext{with the CoV $\gamma=sv$ with Jacobian $d\gamma=\sqrt{v'v}\,ds$}
  &= (v'\Lambda v/2\pi)^{-\half}\,|v|\label{e:para-v}\\
\int_{v^\perp}  \exp\big(-\half\zeta'\Lambda\zeta\big)\,d\zeta
  &= \text{the ratio of \eqref{e:Rd}/\eqref{e:para-v}}\notag\\
  &= |\Lambda/2\pi|^\half(v'\Lambda v/2\pi)^{\half}/|v|, \text{ so}\label{e:perp-v}\\
\int_{v^\perp}  \vpL(\zeta)\,d\zeta
  &= (v'\Lambda v/2\pi)^{\half}/|v|.\label{e:vpL-vperp}\\
\int_{\mu^\|}  \exp\big(-\half w'\Lambda w\big)\,dw
  &=\int_{\bbR} \exp\big(-s^2\mu'\Lambda \mu/2\big)\,|\mu|\,ds\notag\\
  &= (\mu'\Lambda \mu/2\pi)^{-\half}\,|\mu|\label{e:para-mu}\\
\int_{\{v,\mu\}^\perp} \exp\big(-\half q'\Lambda q \big)\,d q
  &= \text{the ratio of \eqref{e:perp-v}/\eqref{e:para-mu}}\notag\\
  &= |\Lambda/2\pi|^\half
     (v'\Lambda v~\mu'\Lambda\mu)^{\half}/2\pi|v|~|\mu|, \text{ so}\notag\\
\int_{\{v,\mu\}^\perp}  \vpL(\zeta)\,d\zeta
  &= (v'\Lambda v~\mu'\Lambda\mu)^{\half}/2\pi|v|~|\mu|.\label{e:v-mu-perp}
\end{align}
If the integral in \eqref{e:para-mu} extends only over those $\zeta\in\mu^\|$
with $\mu'\Lambda\zeta \le \half\log(y_1/y_2)$, its value is reduced by a
factor of
$\Phi\big(\frac{\log(y_1/y_2)-2\mu'\Lambda\mu}
               {2\sqrt{\mu'\Lambda\mu}}\big)$, leading to
\begin{align}
\int_{v^\perp}  \vpL(\zeta-\mu)\,
    \bone{\mu'\Lambda\zeta\ge\half\log(y_1/y_2)}\,d\zeta 
   &= \Phi\Big(\frac{\half\log\frac{y_1}{y_2}-\mu'\Lambda\mu}
                    {\sqrt{\mu'\Lambda\mu}}\Big)
      \Big(\frac{v'\Lambda v}{2\pi v'v}\Big)^\half.  \label{eq:gaussperp}
\end{align}

\section{Additional figures}

\begin{figure}
\captionof{table}{Key indicating identity of currencies in Figure 
\ref{fig:exchrate}.} \label{tab:currency}
\begin{footnotesize}
\begin{center}
\renewcommand{\arraystretch}{.75}
 \begin{tabular}{ccc|ccc}
  col./row \# & symbol & name & col./row \# & symbol & name \\
  \hline
1 & AUD & Australian Dollar & 7 & NLG & Dutch Guilder \\
2 & BEF & Belgian Franc & 8 & NZD & New Zealand Dollar \\
3 & CAD & Canadian Dollar & 9 & ESP & Spanish Peseta \\
4 & FRF & French Franc & 10 & SEK & Swedish Kroner \\
5 & DEM & German Deutschmark & 11 & CHF & Swiss Franc \\
6 & JPY & Japanese Yen & 12 & GBP & British Pound \\
\end{tabular}
\end{center}
\end{footnotesize}
\end{figure}

\begin{figure}
\captionof{table}{Key indicating identity of stocks in the images in Figure 
\ref{fig:djia}} \label{tab:djiacomp}
\begin{center}
\begin{footnotesize}
\renewcommand{\arraystretch}{.75}
 \begin{tabular}{ccc|ccc}
  col./row \# & symbol & name  & col./row \# & symbol & name \\
  \hline
1 & axp & American Express & 16 & mcd & McDonald's \\
2 & ba & Boeing & 17 & mmm & 3M \\
3 & cat & Caterpillar & 18 & mrk & Merck \\
4 & csco & Cisco Systems & 19 & msft & Microsoft \\
5 & cvx & Chevron & 20 & nke & Nike \\
6 & dd & DuPont & 21 & pfe & Pfizer \\ 
7 & dis & Disney & 22 & pg & Proctor \& Gamble \\
8 & ge & General Electric & 23 & t & AT\&T \\
9 & gs & Goldman Sachs & 24 & trv & Travelers \\
10 & hd & Home Depot & 25 & unh & United Healthcare \\
11 & ibm & IBM & 26 & utx & United Technologies \\
12 & intc & Intel & 27 & v & Visa \\
13 & jnj & Johnson \& Johnson & 28 & vz & Verizon \\
14 & jpm & J.P. Morgan Chase & 29 & wmt & Wal-Mart \\ 
15 & ko & Coca-Cola & 30 & xom & Exxon-Mobil \\ 
\end{tabular}
\end{footnotesize}
\end{center}
\end{figure}

\begin{figure}
\centering
\includegraphics[width=0.6\textwidth]{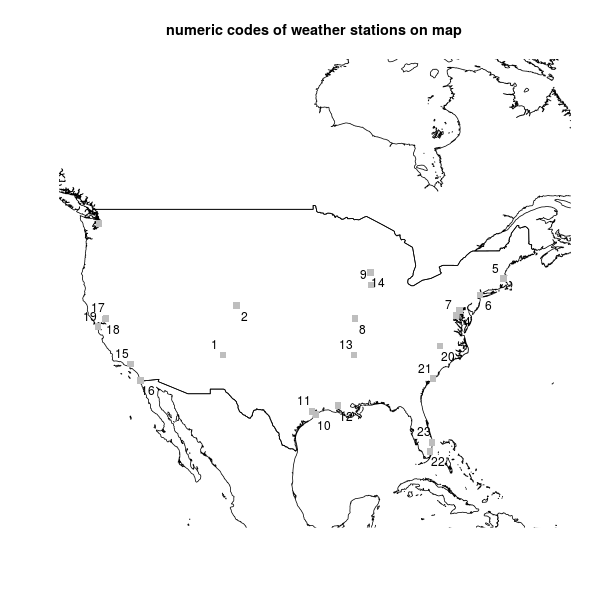}
\caption{Map with numbers labeling locations of weather stations; the numbers correspond to the order in which the stations appear in the colormap images in Figure \ref{fig:weather}.} \label{fig:weathermap}
\end{figure}

\FloatBarrier
\end{appendix}

\bibliographystyle{apalike}
\bibliography{mex-arXiv}

\begin{thebibliography}{}

\bibitem[Balkema and Embrechts, 2007]{balkema2007high}
Balkema, G. and Embrechts, P. (2007).
\newblock {\em High risk scenarios and extremes: a geometric approach}.
\newblock European Mathematical Society.

\bibitem[Ballani and Schlather, 2011]{ballani2011construction}
Ballani, F. and Schlather, M. (2011).
\newblock A construction principle for multivariate extreme value
  distributions.
\newblock {\em Biometrika}, 98(3):633--645.

\bibitem[Beirlant et~al., 2006]{beirlant2006statistics}
Beirlant, J., Goegebeur, Y., Segers, J., and Teugels, J. (2006).
\newblock {\em Statistics of extremes: theory and applications}.
\newblock John Wiley \& Sons.

\bibitem[Buishand et~al., 2008]{buishand2008spatial}
Buishand, T.~A., de~Haan, L., and Zhou, C. (2008).
\newblock On spatial extremes: with application to a rainfall problem.
\newblock {\em The Annals of Applied Statistics}, 2(2):624--642.

\bibitem[Coles, 2001]{coles2001introduction}
Coles, S. (2001).
\newblock {\em An introduction to statistical modeling of extreme values}.
\newblock Springer.

\bibitem[Coles and Tawn, 1991]{coles1991modelling}
Coles, S.~G. and Tawn, J.~A. (1991).
\newblock Modelling extreme multivariate events.
\newblock {\em Journal of the Royal Statistical Society. Series B
  (Methodological)}, 53(2):377--392.

\bibitem[Coles and Tawn, 1996]{coles1996modelling}
Coles, S.~G. and Tawn, J.~A. (1996).
\newblock Modelling extremes of the areal rainfall process.
\newblock {\em Journal of the Royal Statistical Society. Series B
  (Methodological)}, 58(2):329--347.

\bibitem[Coles and Walshaw, 1994]{coles1994directional}
Coles, S.~G. and Walshaw, D. (1994).
\newblock Directional modelling of extreme wind speeds.
\newblock {\em Journal of the Royal Statistical Society, Series C (Applied
  Statistics)}, 43(1):139--157.

\bibitem[Das and Resnick, 2011]{das2011conditioning}
Das, B. and Resnick, S.~I. (2011).
\newblock Conditioning on an extreme component: Model consistency with regular
  variation on cones.
\newblock {\em Bernoulli}, 17(1):226--252.

\bibitem[Davis et~al., 2013]{davis2013max}
Davis, R.~A., Kl{\"u}ppelberg, C., and Steinkohl, C. (2013).
\newblock Max-stable processes for modeling extremes observed in space and
  time.
\newblock {\em Journal of the Korean Statistical Society}, 42(3):399--414.

\bibitem[Davison and Smith, 1990]{davison1990models}
Davison, A.~C. and Smith, R.~L. (1990).
\newblock Models for exceedances over high thresholds.
\newblock {\em Journal of the Royal Statistical Society. Series B
  (Methodological)}, 52(3):393--442.

\bibitem[de~Haan, 1984]{de1984spectral}
de~Haan, L. (1984).
\newblock A spectral representation for max-stable processes.
\newblock {\em The Annals of Probability}, 12(4):1194--1204.

\bibitem[de~Haan and Ferreira, 2006]{de2006extreme}
de~Haan, L. and Ferreira, A. (2006).
\newblock {\em Extreme value theory: an introduction}.
\newblock Springer.

\bibitem[Dzirasa et~al., 2010]{dzirasa2010noradrenergic}
Dzirasa, K., Phillips, H.~W., Sotnikova, T.~D., Salahpour, A., Kumar, S.,
  Gainetdinov, R.~R., Caron, M.~G., and Nicolelis, M. A.~L. (2010).
\newblock Noradrenergic control of cortico-striato-thalamic and mesolimbic
  cross-structural synchrony.
\newblock {\em The Journal of Neuroscience}, 30(18):6387--6397.

\bibitem[Embrechts et~al., 2016]{embrechts2016space}
Embrechts, P., Koch, E., and Robert, C. (2016).
\newblock Space--time max-stable models with spectral separability.
\newblock {\em Advances in Applied Probability}, 48(A):77--97.

\bibitem[Gretton et~al., 2006]{gretton2006kernel}
Gretton, A., Borgwardt, K.~M., Rasch, M., Sch{\"o}lkopf, B., and Smola, A.~J.
  (2006).
\newblock A kernel method for the two-sample-problem.
\newblock In {\em Advances in neural information processing systems}, pages
  513--520.

\bibitem[Gretton et~al., 2012]{gretton2012kernel}
Gretton, A., Borgwardt, K.~M., Rasch, M.~J., Sch{\"o}lkopf, B., and Smola, A.
  (2012).
\newblock A kernel two-sample test.
\newblock {\em The Journal of Machine Learning Research}, 13(1):723--773.

\bibitem[Gumbel, 1960a]{gumbel1960bivariate}
Gumbel, E.~J. (1960a).
\newblock Bivariate exponential distributions.
\newblock {\em Journal of the American Statistical Association},
  55(292):698--707.

\bibitem[Gumbel, 1960b]{gumbel1960distributions}
Gumbel, E.~J. (1960b).
\newblock Distributions des valeurs extr{\^e}mes en plusieurs dimensions.
\newblock {\em Publ. Inst. Statist. Univ. Paris}, 9:171--173.

\bibitem[Harrison and West, 1999]{harrison1999bayesian}
Harrison, J. and West, M. (1999).
\newblock {\em Bayesian Forecasting \& Dynamic Models}.
\newblock Springer.

\bibitem[Heffernan and Resnick, 2007]{heffernan2007limit}
Heffernan, J.~E. and Resnick, S.~I. (2007).
\newblock Limit laws for random vectors with an extreme component.
\newblock {\em The Annals of Applied Probability}, 17(2):537--571.

\bibitem[Heffernan and Tawn, 2004]{heffernan2004conditional}
Heffernan, J.~E. and Tawn, J.~A. (2004).
\newblock A conditional approach for multivariate extreme values (with
  discussion).
\newblock {\em Journal of the Royal Statistical Society: Series B (Statistical
  Methodology)}, 66(3):497--546.

\bibitem[Huser and Davison, 2014]{huser2014space}
Huser, R. and Davison, A.~C. (2014).
\newblock Space--time modelling of extreme events.
\newblock {\em Journal of the Royal Statistical Society: Series B (Statistical
  Methodology)}, 76(2):439--461.

\bibitem[H{\"u}sler and Reiss, 1989]{husler1989maxima}
H{\"u}sler, J. and Reiss, R.-D. (1989).
\newblock Maxima of normal random vectors: between independence and complete
  dependence.
\newblock {\em Statistics \& Probability Letters}, 7(4):283--286.

\bibitem[Minsker et~al., 2017]{minsker2014robust}
Minsker, S., Srivastava, S., Lin, L., and Dunson, D.~B. (2017).
\newblock Robust and scalable {B}ayes via a median of subset posterior
  measures.
\newblock {\em The Journal of Machine Learning Research}, 18(1):4488--4527.

\bibitem[Peres and Sousi, 2015]{peres2015mixing}
Peres, Y. and Sousi, P. (2015).
\newblock Mixing times are hitting times of large sets.
\newblock {\em Journal of Theoretical Probability}, 28(2):488--519.

\bibitem[Pickands, 1981]{pickands1981multivariate}
Pickands, J. (1981).
\newblock Multivariate extreme value distributions.
\newblock In {\em Proceedings 43rd Session International Statistical
  Institute}, volume~2, pages 859--878.

\bibitem[Prado and West, 2010]{prado2010time}
Prado, R. and West, M. (2010).
\newblock {\em Time series: modeling, computation, and inference}.
\newblock CRC Press.

\bibitem[Resnick and Roy, 1991]{resnick1991random}
Resnick, S.~I. and Roy, R. (1991).
\newblock Random {USC} functions, max-stable processes and continuous choice.
\newblock {\em The Annals of Applied Probability}, 1(2):267--292.

\bibitem[Rootz{\'e}n and Tajvidi, 2006]{rootzen2006multivariate}
Rootz{\'e}n, H. and Tajvidi, N. (2006).
\newblock Multivariate generalized {P}areto distributions.
\newblock {\em Bernoulli}, 12(5):917--930.

\bibitem[Schlather, 2002]{schlather2002models}
Schlather, M. (2002).
\newblock Models for stationary max-stable random fields.
\newblock {\em Extremes}, 5(1):33--44.

\bibitem[Schlather and Tawn, 2003]{schlather2003dependence}
Schlather, M. and Tawn, J.~A. (2003).
\newblock A dependence measure for multivariate and spatial extreme values:
  Properties and inference.
\newblock {\em Biometrika}, 90(1):139--156.

\bibitem[Smith, 1984]{smith1984threshold}
Smith, R.~L. (1984).
\newblock Threshold methods for sample extremes.
\newblock In {\em Statistical extremes and applications}, pages 621--638.
  Springer.

\bibitem[Smith, 1990]{smith1990max}
Smith, R.~L. (1990).
\newblock Max-stable processes and spatial extremes.
\newblock {\em Unpublished manuscript, University of Surrey}.

\bibitem[Smola et~al., 2007]{smola2007hilbert}
Smola, A., Gretton, A., Song, L., and Sch{\"o}lkopf (2007).
\newblock Hilbert space embeddings for distributions.
\newblock In {\em Proceedings of the 18th International Conference on
  Algorithmic Learning}, pages 18--31. Springer.

\bibitem[Song et~al., 2009]{song2009hilbert}
Song, L., Huang, J., Smola, A., and Fukumizu, K. (2009).
\newblock Hilbert space embeddings of conditional distributions with
  applications to dynamical systems.
\newblock In {\em Proceedings of the 26th Annual International Conference on
  Machine Learning}, pages 961--968. ACM.

\bibitem[Sriperumbudur et~al., 2010]{sriperumbudur2010hilbert}
Sriperumbudur, B.~K., Gretton, A., Fukumizu, K., Sch{\"o}lkopf, B., and
  Lanckriet, G. R.~G. (2010).
\newblock Hilbert space embeddings and metrics on probability measures.
\newblock {\em The Journal of Machine Learning Research}, 11:1517--1561.

\bibitem[Tawn, 1988]{tawn1988bivariate}
Tawn, J.~A. (1988).
\newblock Bivariate extreme value theory: models and estimation.
\newblock {\em Biometrika}, 75(3):397--415.

\bibitem[Tawn, 1990]{tawn1990modelling}
Tawn, J.~A. (1990).
\newblock Modelling multivariate extreme value distributions.
\newblock {\em Biometrika}, 77(2):245--253.

\bibitem[Welch, 1947]{welch1947generalization}
Welch, B.~L. (1947).
\newblock The generalization ofstudent's' problem when several different
  population variances are involved.
\newblock {\em Biometrika}, 34(1/2):28--35.

\end{thebibliography}

\vfill{\tiny{last edited:\today}}

\end{document}